\documentclass[11pt,reqno]{amsart}
\usepackage{inputenc}
\usepackage{enumerate}
\usepackage{amssymb}
\usepackage{amsmath}
\usepackage{mathrsfs}
\usepackage{amsthm}
\usepackage{tikz-cd}
\usepackage{mathpazo}
\usepackage{extarrows}
\usepackage{color}
\usepackage{setspace}
\usepackage[colorlinks,linkcolor = blue,anchorcolor = red,urlcolor = blue,citecolor = blue]{hyperref}
\usepackage{graphicx, subfig}
\usepackage[square, comma, sort&compress, numbers]{natbib}

\theoremstyle{definition}
\newtheorem{Def}{Definition}[section]
\newtheorem{Exa}[Def]{Examples}
\newtheorem{Rem}[Def]{Remark}

\theoremstyle{plain}
\newtheorem{Thm}[Def]{Theorem}
\newtheorem{Lem}[Def]{Lemma}
\newtheorem{Pro}[Def]{Proposition}
\newtheorem{Cor}[Def]{Corollary}

\newtheorem*{CBC/N}{The relative coarse Baum-Connes/Novikov Conjecture}
\newtheorem*{MCBC/N}{The maximal relative coarse Baum-Connes/Novikov Conjecture}
\newtheorem*{con}{Conjecture}

\newtheorem{innercThm}{Theorem}
\newenvironment{cThm}[1]
  {\renewcommand\theinnercThm{#1}\innercThm}
  {\endinnercThm}

\newtheorem{innercCor}{Corollary}
\newenvironment{cCor}[1]
  {\renewcommand\theinnercCor{#1}\innercCor}
  {\endinnercCor}

\newtheorem{innercDef}{Definition}
\newenvironment{cDef}[1]
  {\renewcommand\theinnercDef{#1}\innercDef}
  {\endinnercDef}

\usepackage[left=2.6cm, right=2.6cm, top=3cm, bottom=3cm]{geometry}
\usepackage[all]{xy}
\usepackage{bm}

\def\Box{{\bf Box}}\def\t{\tau}
\def\IN{\mathbb N}\def\IR{\mathbb R}\def\IC{\mathbb C}\def\IZ{\mathbb Z}

\def\A{\mathcal A}\def\C{\mathcal C}\def\B{\mathcal B}\def\K{\mathcal K}\def\L{\mathcal L}\def\H{\mathcal H}\def\S{\mathcal S}\def\D{\mathcal D}

\def\supp{\textup{supp}}

\def\prop{\textup{Prop}}

\def\Cl{\textup{Cliff}_{\IC}}

\def\ox{\otimes}
\def\wh{\widehat}
\def\wt{\widetilde}
\def\ox{\otimes}

\def\Ga{\Gamma}

\def\ICu{\mathbb C_{Y,\infty}}

\def\CauL{C^*_{L,\max,Y,\infty}}
\def\Cau{C^*_{\max,Y,\infty}}

\author[L.~Guo]{Liang Guo}
\address[L.~Guo]{Shanghai Institute for Mathematics and Interdisciplinary Sciences, Shanghai, 200433, P.~R.~China}
\email{liangguo@simis.cn}

\author[Q.~Wang]{Qin Wang}
\address[Q.~Wang]{Research Center for Operator Algebras, School of Mathematical Sciences, East China Normal University, Shanghai, 200241, P.~R.~China.}
\email{qwang@math.ecnu.edu.cn}

\author[C.~Zhang]{Chen Zhang}
\address[C.~Zhang]{Research Center for Operator Algebras, School of Mathematical Sciences, East China Normal University, Shanghai, 200241, P.~R.~China.}
\email{52275500018@stu.ecnu.edu.cn}

\title[Relative higher index theory and positive scalar curvature at infinity]{Relative higher index theory on quotients of Roe algebras and positive scalar curvature at infinity}
\date{\today}

\begin{document}

\begin{abstract}
In this paper, we employ quotients of Roe algebras as \emph{index containers} for elliptic differential operators to study the existence problem of Riemannian metrics with positive scalar curvature on non-compact complete Riemannian manifolds. The non-vanishing of such an index locates the precise direction at infinity of the obstructions to positive scalar curvature, and may be viewed as a refinement of the positive scalar curvature problem.
To achieve this, we formulate the \emph{relative coarse Baum-Connes conjecture} and the \emph{relative coarse Novikov conjecture}, together with their maximal versions, for general metric spaces as a program to compute the $K$-theory of the quotients of the Roe algebras relative to specific ideals. We show that if the metric space admits a \emph{relative fibred coarse embedding} into Hilbert space or an $\ell^p$-space, certain cases of these conjectures can be verified, which yield obstructions to the existence of uniformly positive scalar curvature metrics in specified directions at infinity.
As an application, we prove that the \emph{maximal coarse Baum-Connes conjecture} holds for finite products of certain expander graphs that fail to admit fibred coarse embeddings into Hilbert space.
\end{abstract}

\maketitle

\tableofcontents

\section{Introduction}\label{sec1}

The question of which complete manifolds admit a Riemannian metric of positive scalar curvature (PSC) stands as one of the most fundamental problems in modern differential geometry. This inquiry forms a crucial bridge between the local, analytic nature of curvature and the global, invariant properties of topology, revealing deep constraints that topology imposes upon the fundamental shape of a space. Its significance extends beyond pure mathematics, touching upon the very fabric of theoretical physics.

Obstructions to the existence of metrics with positive scalar curvature often arise from the topological invariants of the manifold, and a very early obstruction comes from the Lichnerowicz formula and the Atiyah-Singer index theorem. Let $M$ be a compact spin manifold and $D$ the Dirac operator on $M$. The Atiyah-Singer index theorem \cite{ASI1968} establishes that $D$ is a Fredholm operator with $\text{Ind}(D) = \int_M \widehat{A}(M)$. According to the Lichnerowicz formula
\[D^2 = \nabla^*\nabla + \frac{\kappa}{4},\]
where $\kappa$ is the scalar curvature function of the Riemannian manifold $M$, if $\kappa$ is uniformly positive on $M$, then $D$ becomes an invertible operator whose index must be $0$. The non-vanishing of the $\wh A$-genus, as a topological invariant, thus provides an obstruction to the existence of positive scalar curvature.

M. Gromov and H. Lawson proved in \cite{GroLaw1983} that for a non-compact manifold $M$, if there exists a compact subset $Y \subseteq M$ and $\varepsilon_0 > 0$ such that $\kappa|_{Y^c} > \varepsilon_0 > 0$, then $D$ is still a Fredholm operator. But in general, when $M$ is non-compact, $D$ fails to be a Fredholm operator. In \cite{Roe1988}, J. Roe introduced a $C^*$-algebra $C^*(M)$ associated with $M$, later called the Roe algebra of $M$, and proved that $D$ becomes invertible modulo this algebra. Consequently, the ``index'' of $D$ can be interpreted as an element in $K_*(C^*(M))$, which is also called the \emph{higher index} of $D$. 

From the relationship between the Dirac operator and the scalar curvature $\kappa$, the positivity of $\kappa$ typically implies the vanishing of the Dirac operator's index. Therefore, Gromov-Lawson's result can be interpreted as showing that the index information becomes concentrated on the subset $Y$, and the compactness of $Y$ ensures that $D$ is Fredholm.

This phenomenon has been further extended to higher indices. For example, in \cite{XY2014}, Z.~Xie and G.~Yu proved that if $\Ga$ is a countable discrete group with a proper isometric action on a Riemannian manifold $M$, and there exists a $\Ga$-compact subset $Y\subseteq M$ such that the scalar curvature of $M$ is uniformly positive outside $Y$, then the index of $D$ is also concentrated on $Y$, making it an element in $K_*(C^*_r\Ga)$. For the coarse setting, it is also point out in \cite{EW2025} by A.~Engle and C.~Wulff recently if there exists a subset $Y \subset M$ and $\varepsilon_0 > 0$ such that $\kappa_g|_{Y^c} > \varepsilon_0 > 0$, then the index of the Dirac operator $D$ can be interpreted as an element in $K_*(C^*(M, Y))\cong K_*(C^*(Y))$, where $C^*(M, Y)$ is the geometric ideal generated by $Y$. This idea is quite natural since one can expand $Y$ into an $R$-open neighborhood $\text{Pen}(Y, R)$ and $M$ also has positive scalar curvature in $\text{Pen}(Y, R)^c$. Due to the coarse invariance property of Roe algebras, $\text{Pen}(Y, R)$ is coarsely equivalent to $Y$ itself. Therefore, taking the index in the $K$-theory of $C^*(Y)$ and $C^*(\text{Pen}(Y, R))$ is essentially the same. The geometric ideal $C^*(M,Y)$ is defined as $C^*(M,Y) = \lim_{R\to\infty} C^*(\text{Pen}(Y,R))$, which contains all information about subspaces of $M$ that is coarsely equivalent to $Y$, and its $K$-theory is the same as $C^*(Y)$. Hence, interpreting the index as an element of $K_*(C^*(M, Y))$ is actually a more natural choice.
For more recent progress on using the $K$-theory of ideals of Roe algebras as the index containers, see \cite{Roe2016, BL2024, HdK2025}.


These results can be summarized by the principle: \emph{Regions with positive scalar curvature do not carry Dirac index information}. Here we need to briefly explain what we mean by ``a region does not carry Dirac index information''. Consider the following commutative diagram:
\[\begin{tikzcd}
                          & {[D]\in K_*(M)} \arrow[ld, "\text{Ind}_Y"'] \arrow[d, "\text{Ind}"'] \arrow[rd,"\text{Ind}_{Y,\infty}"] &                       \\
{K_*(C^*(M,Y))} \arrow[r, "i_*"] & K_*(C^*(M)) \arrow[r, "\pi_*"]                           & {K_*(C^*(M)/C^*(M,Y))} 
\end{tikzcd}\]
By Roe's work, $\text{Ind}(D)\in K_*(C^*(M))$. And by \cite{EW2025}, we have that the index of $[D]$ actually comes from $K_*(C^*(M,Y))$, thus
$$\text{Ind}(D)=i_*(\text{Ind}_Y(D)),$$
where $\text{Ind}_Y(D)$ is its index in $K_*(C^*(M,Y))$. Since the lower row is exact, we conclude that
$$\text{Ind}_{Y,\infty}(D)=\pi_*(\text{Ind}(D))=\pi_*\circ i_*(\text{Ind}_Y(D))=0.$$
Here, $\text{Ind}_{Y,\infty}(D)$ is a \emph{relative version} of coarse index of $D$. Thus, we say the Dirac index information of $D$ is concentrated on a neighborhood of $Y$ if and only if $\text{Ind}_{Y,\infty}(D)=0$. From the above discussion, we conclude that the non-vanishing of $\text{Ind}_{Y,\infty}(D)$ provides an obstruction to the existence of positive scalar curvature on $Y^c$, thus an obstruction to the existence of positive scalar curvature on $M$. This represents a refined version of the positive scalar curvature existence problem.

Inspired by the above discussion, in this paper, we will employ the quotients of the Roe algebra as \emph{index containers} and study higher index theory for general non-compact metric spaces. This approach has two main advantages: First, the non-vanishing of such an index not only provides an obstruction to the existence of PSC metrics, but also gives a more refined characterization of where exactly this obstruction occurs, serving as a refinement of the original positive scalar curvature problem. On the other hand, for concrete index computations, the quotient index is often easier to compute than the original higher index. For operators whose behavior is difficult to capture directly, we can always quotient out the ideal they generate, after which the $K$-group of the remaining quotient algebra typically becomes much easier to compute.

For a metric space $X$ with bounded geometry and a subset $Y$, we consider a rather special quotient by modulo the \emph{ghostly ideal} associated to this subset, which was introduced in \cite{WZ2023}. The reason for selecting this particular ideal is that we can obtain a clearer view of the asymptotic structure of $Y^c$. Within the ghostly ideal generated by $Y$, there exist ghost operators that vanish at infinity in $Y^c$ yet remain non-zero in the quotient algebra modulo the geometric ideal. Quotienting by the ghostly ideal allows us to better focus on the structure of $Y^c$ at infinity.

As a complement to the above reasoning, we prove the following result:

\begin{cThm}{A}[Theorem \ref{pro: vanishing of the boundary index}]
Let $M$ be a spin, Riemannian manifold with bounded geometry and $Y\subseteq M$ a subspace of $M$. If there exists $\varepsilon_0>0$ such that
$$\liminf_{R\to\infty, x\in \text{Pen}(Y,R)^c}\kappa(x)\geq\varepsilon_0,$$
then for any ideal $I$ containing the geomertic ideal $C^*(M, Y)$, the relative higher index of $D$ with respect to $I$, determined by the composition of the following maps:
$$\text{Ind}_{I}:K_*(M)\xrightarrow{\text{Ind}} K_*(C^*(M))\xrightarrow{\pi_*}K_*(C^*(M)/I),$$
must be $0$, i.e., $\textup{Ind}_{I}([D])=0$.
\end{cThm}

Since ghostly ideal contains the geometric ideal, we have that the quotient algebra obtained by modulo ghostly ideals can still serve as an obstruction to positive scalar curvature. It is worth mentioning that in \cite{EW2025}, A.~Engel and C.~Wulff also utilized the vanishing of indices in quotient algebras to derive obstructions to positive scalar curvature of a submanifold, and their approach primarily involved quotienting by geometric ideals. 

\begin{cDef}{B}[Definition \ref{def: relative Roe algebras}]
The \emph{relative Roe algebra at infinity of $X$ w.r.t. $Y$} is defined to be the quotient algebra of the Roe algebra by the ghostly ideal generated by $Y$. In notation,
$$C^*_{Y,\infty}(X)=C^*(X)/I_G(Y).$$
\end{cDef}

We also introduce the maximal version of the relative Roe algebra. However, there are essential differences between the definitions of the maximal version and the reduced version. The most direct one is that since the maximal representation space is not so concrete, we often cannot write elements of the maximal Roe algebra as $X$-by-$X$ matrices. Therefore, defining ghost ideals in the maximal Roe algebra is not a wise choice. However, the maximal norm is convenient in that its universal property can avoid many troubles. Thus, in the maximal Roe algebra, we have the following definition:

\begin{cDef}{C}[Definition \ref{def: relative Roe algebras}]
The \emph{maximal relative Roe algebra at infinity of $X$ w.r.t. $Y$} is defined to be the quotient algebra modulo the geometric ideal generated by $Y$, i.e.,
$$C^*_{\max,Y,\infty}(X)=C^*_{\max}(X)/C^*_{\max}(X,Y).$$
\end{cDef}


To compute the $K$-theory group of (maximal) relative Roe algebras, we formulate the \textbf{relative coarse Baum-Connes conjecture for $(X,Y)$} and \textbf{relative coarse Novikov conjecture for $(X,Y)$}, together with their maximal counterparts. These assert that the assembly maps
\[\mu_{Y,\infty}: \lim_{d\to\infty} K_{*}(C^{*}_{L,Y,\infty}(P_{d}(X))) \to K_{*}(C^{*}_{Y,\infty}(X))\]
and
\[\mu_{\max,Y,\infty}: \lim_{d\to\infty} K_{*}(C^{*}_{L,\max,Y,\infty}(P_{d}(X))) \to K_{*}(C^{*}_{\max,Y,\infty}(X))\]
are isomorphisms (respectively, injective). Here, the left-hand side $C^*_{L, Y,\infty}(P_d(X))$ is the localization algebra at infinity of $P_d(X)$ relative to $Y$. Its $K$-theory is local so that one can use the Mayer-Vietoris argument to compute it. This provides a method to compute the quotient index of the Dirac operator. We prove that these two conjectures imply the \emph{Gromov-Lawson conjecture}, which claims that uniformly contractible manifolds do not admit a Riemannian metric with uniformly positive scalar curvature:

\begin{cThm}{D}[Theorem \ref{thm: FCE to PSC infty}]
Let $M$ be a uniformly contractible Riemannian manifold with bounded geometry and $Y\subseteq M$ a subspace of $M$. If the (maximal) relative coarse Novikov conjecture holds for $(M, Y)$, then
$$\liminf_{R\to\infty, x\in \text{Pen}(Y,R)^c}\kappa(x)$$
can never be positive. As a result, $M$ can never admit a uniformly positive scalar curvature metric.
\end{cThm}

As for the conjecture, we prove the following theorem:

\begin{cThm}{E}[Theorem \ref{thm: RCBC for RFCE spaces}]\label{thm: main theorem}
If $X$ admits a relative fibred coarse embedding into a Hilbert space with respect to $Y$, then both the relative coarse Baum-Connes conjecture for $(X,Y)$ and its maximal version hold, simultaneously.

Consequently, the canonical quotient map:
$$\pi_*: K_*(C^*_{\max,Y,\infty}(X))\to K_*(C^*_{Y,\infty}(X))$$
is an isomorphism.
\end{cThm}

The concept of \emph{fibred coarse embedding} was introduced by Chen, Wang, and Yu in \cite{CWY2013}, extending Gromov's notion of coarse embedding into Hilbert spaces. In this paper, we propose a concept of \emph{relative fibred coarse embedding} (Definition \ref{def: partial FCE}) as a natural generalization of this framework. This new approach focuses specifically on the fibred coarse embeddability of certain substructures rather than the entire space.

To prove Theorem~\ref{thm: main theorem}, we develop a conceptual framework for adapting the Dirac-dual-Dirac method to this setting. In essence, the proof requires the following two key ingredients:
\begin{itemize}
\item The existence of a \emph{coarsely proper algebra w.r.t. the relative fibred coarse embedding}, and
\item A family of \emph{uniformly flat Bott generators} on this algebra.
\end{itemize}
Under these conditions, we can establish that the relative coarse assembly map is injective. Furthermore, if the Bott generators admit inverses at the level of $K$-theory, then the assembly map becomes an isomorphism. As a direct application of this method, we also prove the following result concerning the relative coarse Novikov conjecture:

\begin{cThm}{F}[Corollary \ref{cor: RCNC for FCE into lp and Hadamard}]
Let $X$ be a metric space that is sparse relative to a subspace $Y\subseteq X$. If $X$ admits a relative fibred coarse embedding into an $\ell^p$-space (with $p\in[1,\infty)$) or a Hadamard manifold with respect to $Y$, then the relative coarse Novikov conjecture for $(X,Y)$ holds.
\end{cThm}

Here, a space being \emph{sparse with respect to $Y$} is a technical condition that roughly means the space is a coarse disjoint union of a sequence of spaces, and the subspace $Y$ embeds as a net into the disjoint union of the first finitely many components.




The relative coarse Baum-Connes/Novikov conjecture can also be used to solve the global coarse Baum-Connes/Novikov conjecture for the whole space $X$:

\begin{cThm}{H}[Theorem \ref{thm: CBC for RFCE}]
Let $X$ be a metric space with bounded geometry which is sparse with respect to $Y\subseteq X$. \begin{itemize}
\item[(1)] If $Y$ satisfies the maximal coarse Baum-Connes conjecture and $X$ admits a relative fibred coarse embedding into Hilbert space w.r.t. $Y$, then the maximal coarse Baum-Connes conjecture holds for $X$.
\item[(2)] If $Y$ satisfies the coarse Novikov conjecture and $X$ admits a relative fibred coarse embedding into an $\ell^p$-space (with $p\in[1,\infty)$) or a Hadamard manifold w.r.t. $Y$, then the coarse Novikov conjecture holds for $X$.
\end{itemize}\end{cThm}

In the above theorem, when deducing the global coarse Baum-Connes conjecture from the relative version, there are differences between the maximal and reduced versions. We can only derive the global conjecture from the relative one in the maximal version, while in the reduced version, we can only deduce the global Novikov conjecture from the relative Novikov conjecture - there is no corresponding surjectivity part.

As byproducts, we can use this to prove the maximal coarse Baum Connes conjecture for the products of a class of expander graphs:

\begin{cCor}{I}[Corollary \ref{cor: FCEH times FCEH}]\label{cor: 1-FCEH times FCEH}
Let $(X_i)_{i=1}^N$ be a finite family of bounded geometry metric spaces that admit a fibred coarse embedding into Hilbert space. Then the maximal coarse Baum-Connes conjecture holds for $\prod_{i=1}^NX_i$.
\end{cCor}

\begin{cCor}{J}[Corollary \ref{cor: CNC for FCEtimesFCE}]
For each $i\in\{1,\cdots, N\}$, assume $X_i$ is a sparse space with bounded geometry, i.e., $X_i=\bigsqcup_{n\in\IN}X_{i,n}$. If $X_i$ admits a fibred coarse embedding into Hilbert space (or a Hadamard manifold, or an $\ell^p$-space with $1\leq p<\infty$). Then the coarse Novikov conjecture holds for $\prod_{i=1}^NX_i$.
\end{cCor}

Let $X=\bigsqcup_{n\in\IN}X_n$ and $Y=\bigsqcup_{m\in\IN}Y_m$ be coarse disjoint unions of finite spaces that admit fibred coarse embeddings into Hilbert space. It is proved in \cite{CWY2013} that the maximal coarse Baum-Connes conjecture holds for both $X$ and $Y$. However, one should notice that $X\times Y=(X_n\times Y_m)_{n,m\in\IN}$ may not admit a fibred coarse embedding into Hilbert space. Corollary \ref{cor: 1-FCEH times FCEH} provides many concrete examples that do not admit fibred coarse embedding into Hilbert space, while whose maximal coarse Baum-Connes conjecture still holds.

The paper is organized as follows: In Section 2, we briefly review the Roe algebra and the ideal structure of the Roe algebra. Using the ideal structure, we define the relative Roe algebra and prove that any quotients of the Roe algebra by an ideal containing a geometric ideal can provide a refined obstruction to positive scalar curvature.
In Section 3, we discuss the localization algebra at infinity and give a definition using exact sequences. With this, we formulate the relative (maximal) coarse Baum-Connes/Novikov conjecture and prove that any of these conjectures can imply the Gromov-Lawson conjecture.
In Section 4, we introduce the conceptual definition of coarsely proper algebras and use them to define the twisted coarse assembly map, proving that this map is an isomorphism. We conceptually introduce what the uniformly flat Bott generators are and prove Theorem \ref{thm: main theorem}. We also consider the relationship between groupoids and relative fibred coarse embeddings.
In Section 5, we explore applications of the relative coarse Baum-Connes/Novikov conjecture. Under certain conditions, it implies the global coarse Baum-Connes/Novikov conjecture for the whole space $X$. As a corollary, we prove that the product of finitely many spaces that admit fibred coarse embeddings into Hilbert space satisfies the maximal coarse Baum-Connes conjecture. We also discuss the sensitivity of choice of subsets.

\subsection*{Acknowledgement}
L.~Guo is partially supported by the Chinese Postdoctoral Science Foundation (No.~2025M773059).
Q.~Wang is partially supported by NSFC (No.~12571135)

\section{Quotients of Roe algebras and obstruction to PSC metrics}\label{sec: Roe algebras and their ideals}

In this section, we present some background on Roe algebras. Since our goal is to develop an index theory using the $K$-theory of quotients of Roe algebras as index containers, we will consequently discuss the dual picture by reviewing key results about the ideal structure of Roe algebras.

\subsection{Roe algebras}\label{section 2.1}

We will specify the notations in this subsection. The reader is referred to \cite{Roe2003} for background information on metric geometry and coarse geometry. 

In this paper, we shall use two distinct notations, $M$ and $X$, to denote proper metric spaces, where $X$ represents a \emph{countable discrete metric space}, and $M$ denotes a general proper metric space. Recall that a metric space $M$ is called \emph{proper} if every closed, bounded subset of $M$ is compact. For $\delta>0$, a \emph{$\delta$-net} $X$ of $M$ is a subset $X\subseteq M$ such that $d(x_1,x_2)\geq\delta$ for any $x_1,x_2\in X$ and $d(m, X)\leq \delta$ for any $m\in M$. We say $M$ has \emph{bounded geometry} if there exists a $\delta$-net $X$ such that $X$ has bounded geometry under the subspace metric, i.e.,
$$\sup_{x\in X}\#\left(B_M(x,R)\cap X\right)<\infty$$
for any $R\geq 0$.  Fix a $\delta$-net $X\subseteq M$. By the axiom of choice, one can take a Borel cover $\{B_x\}_{x\in X}$ of $M$ such that $x\in B_x\subseteq B_M(x,\delta)$ for each $x\in X$.

Let $A$ be a $C^*$-algebra. Take $Z_M\subseteq M$ as a countable \emph{dense} subset of $M$ and $\H=\ell^2(\IN)$ an infinite-dimensional, separable Hilbert space. Let 
$$\H_{M,A}=\ell^2(Z_M)\ox\H\ox A$$ be the Hilbert $A$-module whose $A$-valued inner product is defined by
$$\langle \delta_{z_1}\ox\delta_{n_1}\ox a_1, \delta_{z_2}\ox\delta_{n_2}\ox a_2\rangle=\langle\delta_{z_1},\delta_{z_2}\rangle\cdot\langle\delta_{n_1},\delta_{n_2}\rangle\cdot a_1^*a_2,$$
where $z_1,z_2\in Z_M$, $n_1,n_2\in \IN$ and $a_1,a_2\in A$. Denote $\L_A(\H_{M,A})$ the algebra of all adjointable $A$-module homomorphism on $\H_{M,A}$, it contains $\K_A(\H_{M,A})$, the algebra of all compact $A$-module homomorphism, as an ideal.

 Note that $\H_{M,A}$ admits a canonical representation of the algebra $B(M)$ of all bounded Borel functions on $M$, which is given by
$$f\cdot(\delta_{z_1}\ox\delta_{n_1}\ox a)=f(z_1)\cdot \delta_{z_1}\ox\delta_{n_1}\ox a,\quad\text{ for all }f\in B(M).$$
For each $x\in X$, we denote $\chi_{B_x}$ the characteristic function over $B_x$. For any $T\in\L_A(\H_{M,A})$ and $x,y\in X$, we can define the $(x,y)$-matrix entry of $T$ to be $\chi_{B_x}T\chi_{B_y}: \H_{M,A}\to \H_{M,A}$. In this way, we may view $T$ as a $X$-by-$X$ matrix $T=(T(x,y))_{x,y\in X}$.

\begin{Def}
Let $M$ be a proper metric space, $X\subseteq M$ a $\delta$-net with bounded geometry.
The \emph{algebraic Roe algebra with coefficient $A$}, denoted by $\IC[M,A]$, is defined to be the $*$-subalgebra of $\L_A(\H_{M,A})$ of matrics $T$ satisfying\begin{itemize}
\item[(1)] $T(x,y)\in\K_A(\H_{M,A})$ is compact for any $x,y\in X$;
\item[(2)] there exists $R>0$ such that $T(x,y)=0$ whenever $d(x,y)\geq R$.
\end{itemize}
For an element $T\in\IC[M,A]$, the \emph{propagation} of $T$ is defined to be
$$\prop(T)=\sup\{d(z_1,z_2)\mid \delta_{z_1}T\delta_{z_2}\ne 0, z_1,z_2\in Z_M\}.$$

The \emph{Roe algebra with coefficient $A$}, denoted by $C^*(M,A)$, is defined to be the $C^*$-completion of $\IC[M,A]$ in $\L_A(\H_{M,A})$, and the \emph{maximal Roe algebra with coefficient $A$}, denoted by $C^*_{\max}(X,A)$ is defined to be the $\IC[M,A]$ under the norm
$$\|T\|_{\max}=\sup\{\|\phi(T)\|\mid \phi:\IC[M,A]\to\B(\H_{\phi})\text{ is a $*$-representation}\}.$$
\end{Def}

It is clear that $\IC[M, A]$ is independent of the choice of the $\delta$-net $X$. In this paper, we shall mainly consider the case when $A=\IC$ (the non-trivial coefficient case will be considered when we define the twisted algebra in Section \ref{sec: coarsely proper algebra for FCE}). In this case, we denote $\H_M=\H_{M,\IC}$, and we shall call it the \emph{Roe algebra of $M$}, and denote it by $\IC[M]$, $C^*(M)$, $C^*_{\max}(M)$ for simplicity. Moreover, notice that
$$\big|\prop(T)-\sup\{d(x,y)\mid T(x,y)\ne 0\}\big|\leq 2\delta.$$

\begin{Def}\label{def: K-homology}
The \emph{algebraic localization algebra}, denoted by $\IC_L[M]$, is defined to be the $*$-algebra of all bounded, uniformly continuous functions $g:\IR_+=[0,\infty)\to\IC[M]$ with $\prop(g(t))\to 0$ as $t\to\infty$.

We then define the \emph{localization algebra} $C^*_L(M)$ associated with the norm $\|g\|=\sup_{t\in\IR_+}\|g(t)\|_{\B(\H_M)}$ and the \emph{maximal localization algebra} $C^*_{L,\max}(M)$ associated with the norm $\|g\|_{\max}=\sup_{t\in\IR_+}\|g(t)\|_{\max}$. The \emph{$K$-homology} of $M$ is defined to be
$$K_*(M):=K_*(C^*_L(M))\cong K_*(C^*_{L,\max}(M)).$$
\end{Def}

The localization algebra approach to $K$-homology is first introduced by G.~Yu in \cite{Yu1997}. It is known that the $K$-theory of the maximal localization algebra is isomorphic to that of the reduced localization algebra. They both are $K$-homology of $M$. The reader is also referred to \cite{CWY2013} and \cite{DGWY2025} for some relevant discussion.

The canonical evaluation map $\IC_L[M]\to\IC[M]$ defined by $g\mapsto g(0)$ can be extended to $C^*$-homomorphisms
$$C^*_L(M)\to C^*(M)\quad\text{and}\quad C^*_{L,\max}(M)\to C^*_{\max}(M).$$
These two maps further induce group homomorphisms on the level of $K$-theory
$$\text{Ind}: K_*(M)\to K_*(C^*(M))\quad\text{and}\quad \text{Ind}_{\max}: K_*(M)\to K_*(C^*_{\max}(M)),$$
which are called the \emph{higher index maps}.

In this paper, we shall mainly consider a countable discrete metric space $X$ with \emph{bounded geometry}, which means
$$\sup_{x\in X}\# B(x,R)<\infty$$
for any $R\geq 0$. For such a space, we denote $P_d(X)$ to be the \emph{Rips complex} of $X$ at scale $d$ equipped with \emph{semi-simplicial metric}, which is a proper metric space coarsely equivalent to $X$ for any $d\geq 0$ (see \cite[Section 7.2]{WY2020}).  Since $X$ is a net of $P_d(X)$, we can always view an element in $C^*(P_d(X))$ as an $X$-by-$X$ matrix as above. Moreover, Roe algebra is invariant under coarse equivalence. Thus, the higher index maps defined above lead to the \emph{assembly map} and the \emph{maximal assembly map} as $d$ tends to infinity
$$\mu:\lim_{d\to\infty}K_*(P_d(X))\to K_*(C^*(X))\quad\text{and}\quad \mu_{\max}:\lim_{d\to\infty}K_*(P_d(X))\to K_*(C^*_{\max}(X)).$$
Here comes the \emph{(maximal) coarse Baum-Connes/Novikov conjecture}.

\begin{con}
The coarse Baum-Connes conjecture (coarse Novikov conjecture, resp.) claims the assembly map $\mu$ is an isomorphism (injection, resp.).\\
The maximal coarse Baum-Connes conjecture (maximal coarse Novikov conjecture, resp.) claims the maximal assembly map $\mu_{\max}$ is an isomorphism (injection, resp.).

\end{con}

\subsection{Coarse groupoid and ideals of Roe algebras}

In \cite{STY2002}, Skandalis, Tu, and Yu introduced the notion of \emph{coarse groupoids} for discrete metric spaces with bounded geometry. They proved that the Roe algebra is in fact isomorphic to the crossed product algebra associated to the coarse groupoid. This work established a groupoid approach to the coarse Baum-Connes conjecture. Subsequent research (such as \cite{CW2004, CW2006}) has proved the utility of coarse groupoids in classifying ideals of Roe algebras. 

In this paper, we investigate quotients of Roe algebras, which are dual to ideals of Roe algebras. Naturally, our approach also employs coarse groupoids as a fundamental tool. The reader is referred to \cite{Ren1980, Tu2000} for more details on groupoid crossed product and the Baum-Connes conjecture for groupoids.

Let $X$ be a metric space with bounded geometry. For any $R\geq 0$, the \emph{$R$-diagonal} is defined by
$$\Delta_R=\{(x,y)\in X\times X\mid d(x,y)\leq R\}\subseteq X\times X$$
A subset of a $R$-diagonal is called a \emph{entourage}. The set of all entourages is called the coarse structure of $X$ associated with the metric. Then an operator $T=(T(x,y))_{x,y\in X}\in\mathbb{C}[X]$ has \emph{finite propagation} if and only if the non-zero entries in this matrix form an entourage. Denote by $\beta X$ the Stone-\v{C}ech compactification of $X$, which is the Gelfand spectrum of $\ell^{\infty}(X)$. The \emph{Stone-\v{C}ech boundary} is defined by $\partial_\beta X=\beta X\backslash X$. The \emph{coarse groupoid} of $X$, denoted by $G(X)$, is defined to be
$$G(X)=\bigcup_{R\geq 0}\overline{\Delta_R}\subseteq\beta(X\times X).$$
It was shown in \cite{STY2002} that the coarse groupoid $G(X)$ is locally compact, Hausdorff, $\acute{e}$tale and principal and the space $\beta X$ is exactly the unit space of $G(X)$. The projection onto the first (second, resp.) coordinate $p_1: X\times X\to X$, ($p_2: X\times X\to X$, resp.) extends continuously to the range map $r: G(X)\to\beta X$ (the source map $s: G(X)\to\beta X$, resp.). Readers can refer to \cite{STY2002} or \cite[Chapter 10]{Roe2003} for additional information about coarse groupoids.

The following proposition can be found in \cite[Lemma 4.4]{STY2002}.

\begin{Pro}[\cite{STY2002}]\label{Roe algebra isomorphic to groupoid}
There exists a canonical $*$-isomorphism
$$\theta: C_c(G(X),r^*(\ell^{\infty}(X,\K)))\to \IC[X]$$
which extends to $C^*$--somorphisms
$$\Theta: \ell^{\infty}(X,\K)\rtimes_r G(X)\rightarrow C^{*}(X)\quad\text{and}\quad \Theta_{\max}: \ell^{\infty}(X,\K)\rtimes G(X)\rightarrow C^{*}_{\max}(X).$$
\end{Pro}

This result provides another viewpoint on the Roe algebra based on coarse groupoids. From this perspective, X.~Chen and Q.~Wang conducted extensive research on the ideals of (uniform) Roe algebras, see \cite{CW2004, CW2006}. It turns out that an ideal of (uniform) Roe algebra must correspond to an invariant open set of $\beta X$.Recall that a subset $U\subseteq \beta X$ is $G(X)$-\emph{invariant}, (or invariant for simplicity) if, for any element $\gamma\in G(X)$, it holds that $r(\gamma)\in U$ if and only if $s(\gamma)\in U$.  However, for a given invariant open set, the corresponding ideal may not be unique. For the purposes of this paper, we will primarily focus on ideals associated with invariant open sets generated by a subspace. Moreover, we restrict our consideration to only two types of ideals on such invariant open sets: \emph{geometric ideals} and \emph{ghostly ideals}. 

Recall that a subset $U\subseteq \beta X$ is \emph{invariant} if for any element $\gamma\in G(X)$, one has that $r(\gamma)\in U$ if and only if $s(\gamma)\in U$. Notice that if $U$ is invariant in $\beta X$, the complement $U^c$ is also invariant. Say $Y$ is a subset of $X$, the invariant open set generated by $Y$ is defined to be
$$U_Y=\bigcup_{R\geq 0}\overline{\text{Pen}(Y,R)}\subseteq \beta X.$$
When $Y$ is bounded, the invariant open set generated by $Y$ is exactly $X$. Otherwise, one has that $X\subseteq U_Y\subseteq \beta X$. As a result, $U_Y^c$ is an invariant closed subset of $\beta X$ which is contained in $\partial_\beta X$. For an invariant subset $U\subseteq\beta X$, the restriction of $G(X)$ on $U$ is defined to be
$$G(X)_U=r^{-1}(U),$$
which is a subgroupoid of $G(X)$.

The definition of geometric ideal in Roe algebras is first introduced in \cite{HRY1993}.

\begin{Def}\label{geometric ideal}
The \emph{algebraic geometric ideal} of $\IC[X]$ generated by $Y$, denoted by $\IC[X,Y]$, is the subalgebra of $\IC[X]$ consisting of all operators $T$ whose support is contained in $\text{Pen}(Y,R)\times \text{Pen}(Y,R)$ for some $R>0$.

The operator norm closure of $\IC[X,Y]$ within $C^{*}(X)$ (resp. $C^*_{\max}(X)$) is called \emph{the (maximal) geometric ideal generated by $Y$} and denoted by $C^{*}(X,Y)$ (resp. $C^{*}_{\max}(X,Y)$). 
\end{Def}

The geometric ideal can also be realized in the language of the crossed product of groupoids. For fixed $Y\subseteq X$, we define
$$A_Y=\left\{\xi\in\ell^{\infty}(X,\K) \ \Big|\  \lim_{R\to\infty}\sup_{x\in \text{Pen}(Y,R)^c}\|\xi(x)\|=0\right\}.$$
It is direct to see that $A_Y$ is an ideal of $\ell^\infty(X,\K)$. One has a groupoid description of geometric ideal, which is proved in the uniform Roe algebra setting in \cite{CW2004}. With a similar argument, we also have the crossed product counterpart as the following proposition: 

\begin{Pro}[\cite{CW2004}]\label{geometric ideal and reduced groupoid algebra}
With the definition given above, the $*$-isomorphism $\theta$ in Proposition \ref{Roe algebra isomorphic to groupoid} restricts to a $*$-isomorphism 
$$\theta_Y: C_c(G(X)|_{U_Y},r^*(A_Y))\to \IC[X,Y]$$
which extends to $C^*$-isomorphisms
$$\Theta_Y: A_Y\rtimes_r G(X)|_{U_Y}\to C^*(X,Y)\quad\text{and}\quad \Theta_{Y,\max}: A_Y\rtimes G(X)|_{U_Y}\to C^*_{\max}(X,Y).$$
\qed
\end{Pro}

For any ideal $I$ of the Roe algebra $C^*(X)$, one can detect the corresponding open invariant subset in the following way. For any $\varepsilon>0$ and $T\in I$, the \emph{$\varepsilon$-support} of $T$ is defined by
$$\supp_\varepsilon(T)=\{(x,y)\in X\times X\mid \|T(x,y)\|\geq \varepsilon\}.$$
It is direct to see that $\supp_{\varepsilon}(T)$ is an entourage. Set
$$U(I)=\bigcup_{T\in I,\varepsilon>0}\overline{r(\supp_{\varepsilon}(T)}).$$
If $I=C^*(X,Y)$, the geometric ideal generated by $Y$, then $U(C^*(X,Y))=U_Y$.

The geometric ideal $C^*(X,Y)$ is not the unique ideal corresponding to $U_Y$. The concept of \emph{ghost operators} was named by G.~Yu, which corresponds to counterexamples to the coarse Baum-Connes conjecture. Recently, the \emph{ghostly ideal} was defined by Q.~Wang and J.~Zhang in \cite{WZ2023} for uniform Roe algebras, which can be understood as a local version of ghost operators:

\begin{Def}
Let $X$ be a metric space with bounded geometry, and $Y\subseteq X$ a subset. We define the \emph{ghostly ideal} of the Roe algebra $C^*(X)$ generated by $Y$ to be
$$I_G(Y)=\left\{T\in C^*(X)\ \big|\ \text{for any }\varepsilon>0,\ r(\supp_\varepsilon(T))\subseteq \text{Pen}(Y,R)\text{ for some }R>0\right\}.$$
\end{Def}

It is straightforward to verify that $I_G(Y)$ forms an ideal in the Roe algebra, satisfying $U(I_G(Y)) = U_Y$. From its definition, one can immediately observe that the ghostly ideal $I_G(U_Y)$ is the \emph{largest} among all ideals of the Roe algebra corresponding to the invariant open set $U_Y$. For our purpose of studying quotients of Roe algebras, we note that quotienting by a larger ideal yields a simpler quotient algebra (particularly in terms of its $K$-theory). This explains our particular interest in this ideal.

\begin{Rem}
For uniform Roe algebras, it has been proved in \cite{WZ2023} that the geometric ideal generated by $Y$ is the minimal ideal whose corresponding invariant open set is $U_Y$. However, this does not hold in general for the case of Roe algebras. In \cite{CW2006}, X.~Chen and Q.~Wang show that an ideal of a Roe algebra is determined by a \emph{rank distribution} function on $X\times X$. The rank distribution not only consists of the information of the corresponding open set, but also the rank behavior for each entry. A geometric ideal does not have any restriction on the rank behavior. Therefore, by imposing additional constraints on the rank distribution, we can identify a smaller ideal within the geometric ideal. For example, the uniform rank algebra $UC^*(X)$ introduced by J.~\v{S}pakula is an ideal whose corresponding invariant set is $\beta X$.
\end{Rem}

Fix the invariant open subset $U_{Y}\subseteq\beta X$ generated by $Y$. Denote by $G(X)_{U^{c}_{Y}}:=G(X)\cap s^{-1}(U^{c}_{Y})$ and $G(X)$ can be devided as follows:
\[G(X)=G(X)_{U_{Y}}\cup G(X)_{U^{c}_{Y}}.\]
We have the following short exact sequence:
$${0} \to {C_{c}(G(X)_{U_{Y}},r^{*}(A_{Y}))} \to {C_{c}(G(X),r^{*}(\ell^{\infty}(X,\K)))} \to {C_{c}(G(X)_{U_{Y}^{c}},r^{*}(\ell^{\infty}(X,\K)/A_{Y}))} \to {0}$$

We may complete the sequence in the following two cases:

\textbf{Case 1.} Based on the results discussed above, we have the following sequence of maximal crossed product algebras:

\begin{equation}\label{crossed product sequence}
{0} \to {A_{Y}\rtimes G(X)_{U_{Y}}} \xrightarrow{i} {\ell^{\infty}(X,\K)\rtimes G(X)} \xrightarrow{j} {(\ell^{\infty}(X,\K)/ A_{Y})\rtimes G(X)_{U^{c}_{Y}}} \to {0},
\end{equation}
where the map $i$ is an inclusion and the map $j$ is a restriction. By a similar argument as in \cite[Lemma 2.10]{MRW1996}, one can show that \eqref{crossed product sequence} is \emph{exact}. 

\textbf{Case 2.} Similarly, we can consider the reduced case and obtain the following sequence:
\begin{equation}\label{the reduced sequence}
{0} \to {A_{Y}\rtimes_r G(X)_{U_{Y}}} \xrightarrow{i_r} {\ell^{\infty}(X,\K)\rtimes_r G(X)} \xrightarrow{j_r} {(\ell^{\infty}(X,\K)/ A_{Y})\rtimes_r G(X)_{U^{c}_{Y}}} \to {0},
\end{equation}
By construction, $i_{r}$ is injective, $j_{r}$ is surjective, and $j_{r}\circ i_{r}=0$. However, this sequence is not exact at the middle term in general. The reason lies in that the kernel of $j_{r}$ is usually not the geometric ideal $C^*(X,Y)$, but the ghostly ideal $I_{G}(Y)$, i.e.,
$$\ker(j_{r})=I_G(Y).$$
Readers can refer to \cite{FSN2014} for the proof when the selected subspace $Y$ is bounded, also see \cite{WZ2023} in the context of the uniform Roe algebra.

The notion of relative Roe algebra is introduced in \cite[Definition 6.1]{EW2025}. We shall introduce a similar version of the relative Roe algebra as follows.

\begin{Def}\label{def: relative Roe algebras}
The \emph{(maximal) relative Roe algebra of $X$ w.r.t. $Y$}, denoted by $C^{*}_{Y,\infty}(X)$ (resp. $C^{*}_{\max,Y,\infty}(X)$), is defined to be the quotient algebra
$$C^{*}_{Y,\infty}(X):= (\ell^{\infty}(X,\K)/ A_{Y})\rtimes_r G(X)_{U^{c}_{Y}}\quad\text{and}\quad C^{*}_{\max,Y,\infty}(X):=(\ell^{\infty}(X,\K)/ A_{Y})\rtimes G(X)_{U^{c}_{Y}}.$$
\end{Def}

Together with Proposition \ref{Roe algebra isomorphic to groupoid} and Proposition \ref{geometric ideal and reduced groupoid algebra}, the sequence (\ref{crossed product sequence}), we have the following exact sequences:
$$0 \to {C^{*}_{\max}(X,Y)} \xrightarrow{i}  {C^{*}_{\max}(X)} \xrightarrow{\pi}  {C^{*}_{\max,Y,\infty}(X)} \to 0$$
and
$$0 \to I_G(Y) \xrightarrow{i}  {C^{*}(X)} \xrightarrow{\pi}  {C^{*}_{Y,\infty}(X)} \to {0}.$$
For any $T\in C^*_{Y,\infty}(X)$, we define the propagation of $T$ by
$$\prop(T)=\inf\{\prop(S)\mid S\in C^*(X)\text{ and }\pi(S)=T\}.$$
The maximal case is defined similarly. Parallelly, we shall write $\IC_{Y,\infty}[X]=\IC[X]/\IC[X,Y]$. It is direct to see that $\IC_{Y,\infty}[X]$ is a dense subalgebra of both $C^{*}_{Y,\infty}(X)$ and $C^{*}_{\max,Y,\infty}(X)$.

\begin{Rem}
We should mention that our relative Roe algebra is different from the version of A.~Engel and C.~Wulff in \cite{EW2025}. In their paper, the relative Roe algebra is defined to be the quotient of the Roe algebra $C^*(X)$ by the geometric ideal $C^*(X, Y)$. In this paper, we shall consider instead the quotient by the ghostly ideal rather than the geometric ideal.
\end{Rem}

\subsection{Relative higher index and positive scalar curvature obstruction}

The notion of relative index was first proposed in \cite{GroLaw1983} by Gromov and Lawson. It was originally about two manifolds that are isometric outside a compact set. Then, a gluing technique can be used to define a relative index of the two Dirac operators on a glued manifold as an integer number. If these two Dirac operators are Fredholm, then the relative index of these two manifolds equals the difference between their indices, also see \cite{Roe1991}.

One of the most common ways to guarantee the Fredholmness of the Dirac operator on a non-compact manifold is to assume the manifold has uniform positive scalar curvature outside a compact subset. In \cite{GroLaw1983}, M.~Gromov and H.~Lawson prove that if $M$ is spin and has positive scalar curvature outside a compact subset, then the Dirac operator on $M$ is Fredholm even if $M$ is not compact. This result is generalized by Z.~Xie and G.~Yu in \cite{XY2014} to the equivariant case. If $M$ with a proper $\Ga$-action has positive scalar curvature outside a cocompact set, then the higher index of the Dirac operator can be taken as an element in $K_*(C^*_r\Ga)$. For the coarse geometry case, if we replace the compact set in the theorem of Gromov-Lawson by a unbounded subset $Y$, then the index of the Dirac operator is in the $K$-theory of the geometric ideal $K_*(C^*(X, Y))$, the reader is referred to \cite{EW2025} for more discussions on the coarse case.

Actually, we can generalize the above results to the following theorem:


\begin{Thm}\label{pro: vanishing of the boundary index}
Let $M$ be a spin, Riemannian manifold with bounded geometry and $Y\subseteq M$ a subspace of $M$. If there exists $\varepsilon_0>0$ such that
$$\liminf_{R\to\infty, x\in \text{Pen}(Y,R)^c}\kappa(x)\geq\varepsilon_0,$$
Then for any ideal $I$ containing $C^*(M, Y)$, the relative higher index of $D$ with respect to $I$, determined by the composition of the following maps:
$$\text{Ind}_{I}:K_*(M)\xrightarrow{\text{Ind}} K_*(C^*(M))\xrightarrow{\pi_*}K_*(C^*(M)/I),$$
must be $0$, i.e., $\textup{Ind}_{I}([D])=0$.
\end{Thm}

\begin{proof}
The proof is split into the following two steps.

\noindent\emph{Step 1.} Let $N\subseteq M$ be an open subset of $M$ such that $N\cup Y$ is coarsely equivalent to $M$. Then $C^*(M)/I\cong C^*(N)/C^*(N)\cap I$.

Set $Z$ to be a net of $M$ which has bounded geometry such that $Z_N\cup Z_Y=Z$, where $Z_N=Z\cap N$ and $Z_Y=Z\cap Y$. Such a net exists since $M$ is coarsely equivalent to $N\cup Y$, and one can always find such a net in $N\cup Y$ by bounded geometry. Since Roe algebras are coarsely invariant, we shall identify $C^*(M)$ (resp., $(C^*(M,N)), C^*(M,Y)$) with $C^*(Z)$ (resp., $(C^*(Z,Z_N)), C^*(Z,Z_Y)$).

By the isomorphism theorem, we have that
$$\frac{C^*(N)+I}{I}\cong\frac{C^*(N)}{C^*(N)\cap I}$$
Thus, it suffices to show that $C^*(N)+I=C^*(M)$. For any $T\in C^*(M)$, we choose a sequence $\{S_n\}_{n\in\IN}\subseteq\IC[M]$ such that $\|T-S_n\|\to 0$. Define $S_{N,n}=\chi_{Z_N}S_n\chi_{Z_N}\in\IC[N]$. Since $Z_N\cup Z_Y=Z$, we conclude that $\supp(S-S_{N,n})\subseteq B(Z_Y,\prop(S_n))$. Thus $S_n-S_{N,n}\in C^*(M,Y)$. To sum up, we have that $C^*(N)+C^*(M,Y)=C^*(M)$. Let $n$ tend to $\infty$. We conclude that
$$\chi_{Z_N}T\chi_{Z_N}=\lim_{n\to\infty}\chi_{Z_N}S_n\chi_{Z_N}\in C^*(N)\quad\text{and}\quad T-\chi_{Z_N}T\chi_{Z_N}\in C^*(M,Y),$$
i.e., $C^*(N)+C^*(M,Y)=C^*(M)$. Since $I$ contains $C^*(M,Y)$, we have that $C^*(N)+I=C^*(M)$.

\noindent{\emph{Step 2.}}
Since $$\liminf\limits_{R\to\infty, x\in \text{Pen}(Y,R)^c}\kappa(x)\geq\varepsilon_0,$$ there exists $R>0$ such that $\kappa_g(x)\geq\frac{\varepsilon_0}2>0$ for all $x\in \text{Pen}(Y,R)^c$. Set $N=\overline{\text{Pen}(Y,R)}^c$. It is clear to see that $N$ is an open subset satisfying that $N\cup Y$ is coarsely equivalent to $M$. Denote by $[D]$ the $K$-homology class determined by $D$ in $K_*(M)$. Then one can restrict $D$ on $N$ and we denote by $[D_N]$ the corresponding $K$-homology class in $K_*(N)$. With Step 1, we then have the following commuting diagram:
$$\begin{tikzcd}
K_*(N) \arrow[r, "\textup{Ind}"]                  & K_*(C^*(N)) \arrow[d, "ad_{V_i}*"] \arrow[rd] \\
K_*(M) \arrow[r, "\textup{Ind}"] \arrow[u, "i^*"] & K_*(C^*(M)) \arrow[r]                         & {K_*(C^*(M)/I),}
\end{tikzcd}$$
where $i: N\to M$ is the canonical inclusion, which is both smooth and coarse, and $i^*$ is the induced $K$-homology map, and $V_i$ is the covering isometry for $i$. Since the metric on $N$ has strictly positive scalar curvature, $\textup{Ind}([D_N])=0$ by the Lichnerowicz formula. By definition, we have that $i^*([D])=[D_N]$. Then the proposition followes by a diagram chasing argument.
\end{proof}


\begin{Rem}
Note that for any ideal $I$ of the Roe algebra, we can always push forward the index of the Dirac operator to the quotient algebra via the composition of maps
$$\text{Ind}_{I}: K_*(M)\xrightarrow{\text{Ind}} K_*(C^*(M))\xrightarrow{\pi_*}K_*(C^*(M)/I).$$
As seen from the above proposition, if the ideal contains a geometric ideal $C^*(M,Y)$, then the non-vanishing of the quotient index corresponds to an obstruction to the existence of uniformly positive scalar curvature in $Y^c$ at infinity.
\end{Rem}

In the previous subsection, we have defined the relative Roe algebra as the quotient of the Roe algebra by the ghostly ideal. For a subset $Y\subseteq M$, the ghostly ideal $I_G(Y)$ is maximal and contains the corresponding geometric ideal $C^*(X, Y)$. According to the above proposition, the non-vanishing of the index in the relative Roe algebra can reflect obstructions to positive scalar curvature on $Y^c$, which is also an obstruction to positive scalar curvature on the whole $M$. This index is called the \emph{relative index associated with $Y$}
$$\text{Ind}_{Y,\infty}: K_*(M)\xrightarrow{\text{Ind}} K_*(C^*(M))\xrightarrow{\pi_*}K_*(C^*_{Y,\infty}(M)).$$
On the other hand, quotienting by the ghostly ideal is a more practical choice for index computations. This is because the calculation of $K$-theory for ghostly ideals is typically much more challenging than for geometric ideals. For instance, when $Y$ is bounded, the geometric ideal reduces to the compact operator algebra whose $K$-theory is well-known, while the $K$-theory of the ghostly ideal in this case is considerably more complicated if $M$ does not have Property A (see \cite{WZ2023}). Therefore, by quotienting out the more complex object, we obtain a quotient algebra whose $K$-theory is relatively easier to compute.

\section{Relative higher index theory}

In this section, we present a method for computing the $K$-groups of relative Roe algebras through the \emph{assembly maps}. The underlying idea is rather straightforward. Indeed, in the previous section we have already realized the relative Roe algebra as a certain groupoid crossed product algebra. According to \cite{Tu2000}, we can always define a groupoid Baum-Connes assembly map for boundary groupoids, thereby transforming the computation of the $K$-theory of Roe algebras into computations of certain algebraic-topological invariants, namely the $K$-homology of classifying spaces.
From our perspective, while the groupoid language is powerful, we prefer to understand Roe algebras through the more concrete metric space viewpoint rather than the abstract groupoid $KK$-theory framework. Therefore, we introduce a notion of \emph{relative $K$-homology} by using localization algebras, which then leads to the construction of a \emph{relative coarse assembly map}.

\subsection{Relative coarse assembly maps}

We start with some discussion on the localization algebras.

\begin{Def}
Let $M$ be a proper metric space, $Y\subseteq M$ a subset. \emph{The localization algebra of $M$ localized near $Y$}, denoted by $C^{*}_{L, Y}(M)$, is defined to be the subalgebra of $C^*_L(M)$ of all bounded and uniformly norm-continuous functions $g:\IR_+\to C^*(M,Y)$.

The \emph{$K$-homology of $M$ localized near $Y$} is defined to be
$$R_YK_*(M)=K_*(C^*_{L,Y}(M)).$$
\end{Def}

Since the maximal and reduced localization algebras are equivalent at the $K$-theory level, one can also define the localization algebra localized near $Y$ in the maximal norm, denoted by $C^*_{L,\max,Y}(M)$, which will give the same $K$-homology. Moreover, one can also replace the geometric ideal by the ghostly ideal, denoted by $C^*_{L,Y}(M,I_G(Y))\subseteq C_{ub}(\IR_+,I_G(Y))$. As we need the propagation tend to $0$ in localization algebras, the asymptotic behavior of $C^*_{L,Y}(M)$ and $C^*_{L,Y}(M,I_G(Y))$ would be the same since the finite propagation part of the geometric ideal coincides with that of the ghostly ideal. After taking $K$-theory, they are the same. From the definition, it is easy to see that $R_YK_*(M)$ is local in $M$ and coarse in $Y$, More precisely, if there is a pair $(M',Y')$ and a Lipschitz (more generally, continuous coarse) map $f: M'\to M$ such that $f(Y')\subseteq Y$, then it induces a group homomorphism
$$f_*: R_{Y'}K_*(M')\to R_YK_*(M).$$
If $f$ is a strong Lipschitz homotopy equivalence and $f(Y')$ is coarsely equivalent to $Y$, then $f_*$ is an isomorphism.

\begin{Lem}\label{lem: representable K-homology}
Let $X$ be a metric space with bounded geometry, and $Y$ a subspace of $X$. Then for any $d\geq 0$, we have that $R_YK_*(P_d(X))$ is isomorphic to $\lim\limits_{R\to\infty}K_*(P_d(\text{Pen}(Y,R)))$.
\end{Lem}

\begin{proof}
We shall identify $C^*(P_d(X),Y)$ with $\lim_{R\to\infty}C^*(P_d(\text{Pen}(Y,R)))$ as in \cite{HRY1993}. Thus, there is a canonical inclusion $C^*_L(P_d(\text{Pen}(Y,R)))\to C^*_{L,Y}(P_d(X))$ for any $R\geq 0$, which induces an inclusion
\begin{equation}\label{eq: localization algebras}i_*: \lim_{R\to\infty}K_*(C^*_L(P_d(\text{Pen}(Y,R))))\to K_*(C^*_{L,Y}(P_d(X))).\end{equation}
We should mention that both sides admit the Mayer-Vietoris sequence. For the left side, the proof is similar to \cite[Proposition 9.4.13]{WY2020}. For the right side, as we discussed above, $R_YK_*(M)$ is local in $M$. The detailed proof is similar to \cite[Proposition 3.11]{Yu1997}. Thus, using a cutting and pasting argument as in \cite[Theorem 3.2]{Yu1997}, it suffices to prove it for the $0$-skeleton. In this case, the left side becomes
$$K_*\left(\lim_{R\to\infty}\prod_{x\in \text{Pen}(Y,R)}C^*_{ub}(\IR_+,\K(\H_x))\right)=\lim_{R\to\infty}\prod_{x\in \text{Pen}(Y,R)}K_*\left(C^*_{ub}(\IR_+,\K(\H_x))\right),$$
where $\H_x=\delta_x\ox \H\subseteq\ell^2(X)\ox \H$. For the right-hand side, it becomes
$$K_*\left(C^*_{ub}\left(\IR_+,\lim_{R\to\infty}\prod_{x\in \text{Pen}(Y,R)}\K(\H_x)\right)\right)=\lim_{R\to\infty}\prod_{x\in \text{Pen}(Y,R)}K_*\left(\K(\H_x))\right).$$
It is direct to see that $\lim\limits_{R\to\infty}\prod\limits_{x\in \text{Pen}(Y,R)}\K(\H_x)$ is quasi-stable and the identity above follows from \cite[Lemma 12.4.3]{WY2020}. As a result, the map \eqref{eq: localization algebras} is an isomorphism restricting to $0$-skeleton. This finishes the proof.
\end{proof}

As a direct corollary, if $Y$ is bounded, then $R_YK_*(P_d(X))$ is isomorphic to the representable $K$-homology $RK_*(P_d(X))$.

\begin{Def}\label{def: localization algebra at infinity}
\emph{The localization algebra at infinity of $M$ relative to $Y$}, denoted by $C^*_{L,Y,\infty}(M)$, is defined to be the quotient algebra $C^*_L(M)/C^*_{L,Y}(M,I_G(Y))$.

The maximal version is defined to be $C^*_{L,\max,Y,\infty}(M)=C^*_{L,\max}(M)/C^*_{L,\max,Y}(M)$.
\end{Def}

As we discussed above, $C^*_{L,Y,\infty}(M)$ and $C^*_{L,\max,Y,\infty}(M)$ has the same $K$-theory. By definition, we also have the following six-term exact sequence
$$\begin{tikzcd}
R_YK_0(M) \arrow[r] & K_0(M) \arrow[r] & K_0(C^*_{L,Y,\infty}(M)) \arrow[d] \\
K_1(C^*_{L,Y,\infty}(M)) \arrow[u] & K_1(M) \arrow[l] & R_YK_1(M) \arrow[l]
\end{tikzcd}$$

An element in $C^*_{L,Y,\infty}(M)$ can be seen as a bounded uniformly continuous function $f:\IR_+\to C^*_{Y,\infty}(M)$. Thus, there exists a canonical evaluation map
$$ev:C^*_{L,Y,\infty}(M)\to C^*_{Y,\infty}(M),\quad\text{defined by } f\mapsto f(0)$$
which induces a relative index map
$$\text{Ind}_{Y,\infty}: K_*(C^*_{L,Y,\infty}(M))\to K_*(C^*_{Y,\infty}(M)).$$
Now, we can formulate the relative version of the coarse Baum-Connes conjecture. Let $X$ be a metric space with bounded geometry, $Y\subseteq X$ a subspace. Then for any $d\geq 0$, we can define the relative index map
$$\text{Ind}_{Y,\infty}: K_*(C^*_{L,Y,\infty}(P_d(X)))\to K_*(C^*_{Y,\infty}(X)).$$
The relative coarse assembly map is defined by pushing $d$ to infinity:
$$\mu_{Y,\infty}: \lim_{d\to\infty}K_*(C^*_{L,Y,\infty}(P_d(X)))\to K_*(C^*_{Y,\infty}(X)).$$

\begin{CBC/N}
Let $X$ be a metric space with bounded geometry, $Y$ a subspace of $X$. \begin{itemize}
\item The \textbf{relative coarse Baum-Connes conjecture} for $(X,Y)$: the relative coarse assembly map $\mu_{Y,\infty}$ is an isomorphism;
\item The \textbf{relative coarse Novikov conjecture} for $(X,Y)$: the relative coarse assembly map $\mu_{Y,\infty}$ is injective.
\end{itemize}\end{CBC/N}

One can similarly define the maximal relative coarse assembly map
$$\mu_{\max, Y,\infty}: \lim_{d\to\infty}K_*(C^*_{L,\max,Y,\infty}(P_d(X)))\to K_*(C^*_{\max,Y,\infty}(X)).$$
Thus, we have the following maximal analogue of the relative coarse Baum-Connes conjecture.

\begin{MCBC/N}
Let $X$ be a metric space with bounded geometry, $Y$ a subspace of $X$. \begin{itemize}
\item The \textbf{maximal relative coarse Baum-Connes conjecture} for $(X,Y)$: the maximal relative coarse assembly map $\mu_{\max, Y,\infty}$ is an isomorphism;
\item The \textbf{maximal relative coarse Novikov conjecture} for $(X,Y)$: the maximal relative coarse assembly map $\mu_{\max, Y,\infty}$ is an injection.
\end{itemize}\end{MCBC/N}

The relative coarse Baum-Connes conjecture has two special cases. When the subset $Y = \emptyset$, the relative coarse Baum-Connes conjecture reduces to the global coarse Baum-Connes conjecture. When $Y$ is coarsely equivalent to $X$, the relative Roe algebra becomes trivial, making the relative coarse Baum-Connes conjecture vacuously true. When $Y$ is bounded, the relative coarse Baum-Connes conjecture coincides with the boundary coarse Baum-Connes conjecture introduced by M.~Finn-Sell and N.~Wright \cite{FSN2014}.

\subsection{Existence of positive scalar curvature at infinity}

The non-existence of positive scalar curvature (abbreviated as the PSC problem) is a classic problem in differential geometry. This problem first originated with the torus $\mathbb{T}^n$ - whether there exists a Riemannian metric $g$ on $\mathbb{T}^n$ with uniformly positive scalar curvature bounded below by some $\varepsilon_0 > 0$. In \cite{SY1979}, for dimensions $n \leq 7$, Schoen and Yau used the minimal surface method to prove that no such metrics exist on $\mathbb{T}^n$. Later in \cite{GroLaw1983}, Gromov and Lawson extended this to all dimensions using higher index theory. Recently, Schoen and Yau removed the dimensional restrictions using minimal surface methods in \cite{SY2022}. These results led to the Gromov-Lawson conjecture \cite{Rosenberg1983}, which states that no positive scalar curvature metrics exist on aspherical manifolds. Notably, the coarse Baum-Connes conjecture implies this conjecture. The reader is referred to \cite{Rosenberg2007} for a survey on the development of the PSC problem.

\begin{Thm}\label{thm: FCE to PSC infty}
Let $M$ be a uniformly contractible Riemannian manifold with bounded geometry and $Y\subseteq M$ a subspace of $M$. If the (maximal) relative coarse Novikov conjecture holds for $(M,Y)$, then
$$\liminf_{R\to\infty, x\in \text{Pen}(Y,R)^c}\kappa(x)$$
can never be strictly positive. As a result, $M$ can never admit a uniformly positive scalar curvature metric.
\end{Thm}

We should remark here that Theorem \ref{thm: FCE to PSC infty} provides a refinement of the Gromov-Lawson conjecture on the existence problem of positive scalar curvature on a non-compact complete Riemannian manifold. Here, we do not consider the scalar curvature of the entire manifold, but only consider the asymptotic behavior of the scalar curvature towards infinity at some chosen directions. The validity of Theorem \ref{thm: FCE to PSC infty} tells us that the vanishing of the \emph{relative index} of the Dirac operator not only provides an obstruction to the Gromov-Lawson conjecture, but also can precisely identify from which direction at infinity this obstruction arises.

\begin{proof}[Proof of Theorem \ref{thm: FCE to PSC infty}]
Here, we only prove for the reduced case, the maximal case follows from a similar argument. Take $X\subseteq M$ to be a net with bounded geometry and $Z=X\cap Y$. We claim that if $M$ is uniformly contractible, then $K_*(C^*_{L,Y,\infty}(M))$ is isomorphic to $\lim_{d\to\infty}K_*(C^*_{L,Z,\infty}(P_d(X)))$. The idea of the proof of the above claim is similar with \cite[Theorem 7.3.6]{WY2020}, which proves that for any uniformly contractible metric space, $K_*(M)$ is isomorphic to $\lim_{d\to\infty}K_*(P_d(X))$. By the six-term exact sequence, it suffices to show that $R_YK_*(M)$ is isomorphic to $\lim_{d\to\infty}R_ZK_*(P_d(X))$.

As proved in \cite[Theorem 7.3.6]{WY2020}, one can always find a continuous map $f: P_d(X)\to M$ for any $d\geq 0$ which restricts a coarse equivalence between $Z$ and $Y$. On the other side, there also exists a sufficiently large $d$ such that one can find a continuous map $g: M\to P_d(X)$ such that the composition $f\circ g$ is homotopy equivalent to the identity map. Moreover, $g\circ f: P_d(X)\to P_d(X)\xrightarrow{i}P_{d'}(X)$ is homotopy equivalent to the canonical inclusion map $i: P_d(X)\to P_{d'}(X)$ for sufficiently large $d'$. As we discussed before, $f$ and $g$ induce group homomorphisms which are inverse to each other:
$$f_*: R_YK_*(M)\to\lim_{d\to\infty}R_ZK_*(P_d(X))\quad\text{and}\quad g_*:\lim_{d\to\infty}R_ZK_*(P_d(X))\to R_YK_*(Y).$$
This proves the claim.

Assume for a contradiction that one can find a Riemannian metric $g$ on $M$ such that there exists $\varepsilon_0$ such that
$$\liminf_{R\to\infty, x\in \text{Pen}(Y,R)^c}\kappa_g(x)>2\varepsilon_0>0.$$
Here $\text{Pen}(Y,R)$ is the \emph{closed} $R$-neighbourhood of $Y$. Then there exists $R\geq 0$ such that $\kappa_g(x)\geq\varepsilon_0$ for all $x\in \text{Pen}(Y,R)^c$. Take $D$ to be the Dirac operator on $M$ associated with the spin structure and the Riemannian structure of $M$, then by Lichnerowicz's formula
\begin{equation}\label{eq: Lichnerowicz}D^*D=\nabla^*\nabla+\frac{\kappa_g}4,\end{equation}
where $\nabla^*\nabla$ is the Bochner-Laplacian operator, which is positive. Denote $M_R=\text{Pen}(Y,R)^c$ for simplicity, which is an open submanifold of $M$, then the Dirac operator $D$ restricts to the Dirac operator $D_R$ on $M_R$. Since $\nabla^*\nabla$ is positive and $\kappa_g$ is strictly positive on $M_R$, $D_R$ is invertible, whose higher index should be $0$. Notice that $[D_R]\in K_*(M_R)$ naturally descends to an element $\pi([D_R])\in K_*(C^*_{L,Y,\infty}(M))$ whose index in $K_*(C^*_{Y,\infty}(M))$ should also be $0$. Since the relative coarse assembly map
$$\mu_{Y,\infty}: K_*(C^*_{L,Y,\infty}(M))\to K_*(C^*_{Y,\infty}(M))$$
is an isomorphism, this implies that $\pi([D_R])$ on the $K$-homology side should be $0$. By Step 1 of Proposition \ref{pro: vanishing of the boundary index}, we have that $[D_R]\in K_*(M_R)$ and $[D]\in K_*(M)$ have the same relative index.

By the six-term exact sequence, one has that $[D]\in K_*(M)$ is actually in the image of $R_YK_*(M)$. By Lemma \ref{lem: representable K-homology}, there exists $S>0$ such that $[D]$ is in the image of $K_*(\text{Pen}(Y,S))$. Take a open subset $U$ in $\text{Pen}(Y,S)^c$ which is homeomorphic to $\IR^n$ and some non-zero $K$-theory element $\alpha$ in $K^0(M)$ whose support lies in $U$, then the pairing of $[D]$ and $\alpha$ should never be $0$ since $[D]$ restricts to the generator of $K_0(U)$. However, since $[D]$ is in the image of $K_*(\text{Pen}(Y,S))$, the support of $[D]$ has empty intersection with the support of $\alpha$. This leads to a contradiction.
\end{proof}

As a corollary, we have the following result.

\begin{Cor}\label{cor: scalar curvature problem}
Let $M$ be a closed spin aspherical manifold, $\wt M$ the universal cover of $M$. Then the (maximal) relative coarse Novikov conjecture of $(\wt M, Y)$ for any $Y\subseteq \wt M$ which is not coarse equivalent to $\wt M$ implies the Gromov-Lawson conjecture of $M$.
\end{Cor}

\subsection{A remark on general ghostly ideals}

Although this paper primarily focuses on ghostly ideals generated by a single subspace $Y$, we should mention that for general ghostly ideals in Roe algebras, one can still replicate the preceding constructions to develop the corresponding relative higher index theory.

Let $U\subseteq \beta X$ be an invariant open subset in the coarse groupoid $G(X)$. The ghostly ideal of $C^*(X)$ associated with $U$ has been defined as
$$I_G(U)=\{T\in C^*(X)\mid \text{for any $\varepsilon$, }\overline{r(\supp_{\varepsilon}(T))}\subseteq U\}.$$
The localization algebra localized in $U$ of $P_d(X)$ can be similarly defined, denoted by $C^*_{L,U}(P_d(X),I_G(U))$, to be the subalgebra of $C^*_L(P_d(X))$ of all functions that only take values in $I_G(U)$. Thus, we also have the relative coarse assembly map between the two quotient algebras similar to before:
$$\mu_{U,\infty}: \lim_{d\to\infty}K_*\left(\frac{C^*_L(P_d(X))}{C^*_{L,U}(P_d(X),I_G(U))}\right)\to K_*\left(\frac{C^*(X)}{I_G(U)}\right).$$
We shall call $\mu_{U,\infty}$ the coarse relative assembly map for $(X,U)$.

As discussed in \cite[Definition 6.1]{CW2004}, the notion of \emph{ideals $\bf L$ of a space} $X$ is defined as a family of subsets of $X$ satisfying that\begin{itemize}
\item if $Y\in\textbf{L}$ and $Z\subseteq Y$; then $Z\in\textbf{L}$;
\item if $Y,Z\in\textbf{L}$, then $Y\cup Z\in\textbf{L}$;
\item if $Y\in\textbf{L}$, then $\text{Pen}(Y,R)\in\textbf{L}$.
\end{itemize}An invariant open set $U$ determines such an ideal by
$$\textbf{L}(U)=\{Y\subseteq X\mid \overline{Y}\subseteq U\subseteq \beta X\},$$
see \cite[Proposition 6.2]{CW2004}. We define a \emph{filter} $\omega_U$ on the set $X$ by taking the complement of $\textbf{L}(U)$, i.e.,
$$Y\in\omega_U\iff Y^c\in \textbf{L}(U).$$
It is direct from the definition that $\omega_U$ is a filter. Moreover, this filter is \emph{coarse} in the sense of \cite[Section 4.2]{Georgescu}, i.e.,
$$Y\in\omega_U\iff Y^{(r)}\in \omega_U\text{ for any $r>0$},$$
where $Y^{(r)}=\{x\in X\mid d(x,Y^c)> r\}$.

Let $M$ be a uniformly contractible manifold and $X\subseteq M$ a net with bounded geometry. For a filter $\omega$ on $X$, we define the its close extension $\overline{\omega}$ on $M$ to be
$$\overline{\omega}=\{Z\subseteq M\mid Z\text{ is close, }Z\cap X\in \omega\}.$$
For a real function $f: M\to\IR$ and, we define the limit inferior of $f$ associated with $\overline{\omega}$ to be
$$\liminf_{x\to\overline{\omega}}f(x)=\sup_{Y\in\overline{\omega}}\inf_{x\in Y}f(x).$$
One should notice that, for a subset $Y\subseteq X$,
$$\liminf_{R\to\infty, x\in \text{Pen}(Y,R)^c}f(x)=\liminf_{x\to\overline{\omega_{U_Y}}}f(x).$$
Let $U\subseteq \beta X\subseteq G(X)$ be an invariant open subset. One can similarly with Theorem \ref{thm: FCE to PSC infty} prove that if the relative coarse assembly map $\mu_{U,\infty}$ is injective, then $\liminf_{x\to\overline{\omega_{U}}}\kappa(x)$ can never be positive.

In fact, the conclusions presented later in the whole of this paper can also be generalized to this more general setting. However, the case of invariant open sets generated by a single set $Y$ offers stronger intuitiveness and is already sufficiently enlightening. Therefore, in the subsequent discussion, we will only consider invariant open sets generated by a single set $Y$.

\section{Relative fibred coarse embeddings}\label{sec: coarsely proper algebra for FCE}

In this section, we introduce a relative version of \emph{fibred coarse embedding} as a generalization of the fibred coarse embedding. Additionally, we define the concept of \emph{coarsely proper algebras} and study the \emph{twisted assembly map} with coefficients in these algebras. Our main technical result establishes that the twisted assembly map is always an isomorphism. This construction will provide a framework for proving the relative coarse Baum-Connes conjecture. Moreover, we also provide a conceptual description of \emph{uniformly flat Bott generators} for coarsely proper algebras. Under this framework, we prove the relative coarse Baum-Connes/Novikov conjecture in several cases.

\subsection{Relative fibred coarse embedding}\label{section 3.1}
In this subsection, we will introduce the concept of relative fibred coarse embedding, which serves as a generalization of the fibred coarse embedding into Hilbert space in \cite{CWY2013}.

\begin{Def}[\cite{CWY2013}]\label{def: FCE}
A  metric space $(X,d)$ is said to admit a fibred coarse embedding into Hilbert space $H$ if there exist \par 
$\bullet$ a field of Hilbert spaces $(H_{x})_{x\in X}$ over $X$;\par
$\bullet$ a section $s:X\rightarrow\bigsqcup_{x\in X}H_{x}$ ($\text{i.e.}\ s(x)\in H_{x}$);\par 
$\bullet$ two non-decreasing functions $\rho_{1}$ and $\rho_{2}$ from $\mathbb{R}_{+}$ to $\mathbb{R}_{+}$ with $\lim\limits_{r\rightarrow\infty}\rho_{i}(r)=\infty\ (i=1,2)$\\
such that for any $r>0$ there exists a bounded subset $K\subset X$ for which there exists a ``trivialization"
$t_{C}:(H_{x})_{x\in C}\rightarrow 
C\times H$ for each subset $C\subset X\setminus K$ of diameter less than $r$, i.e. a map from $(H_{x})(x\in C)$ to the constant field $C\times H$ over $C$ such that the restriction of $t_{C}$ to the fiber $H_{x}(x\in C)$ is an affine isometry $t_{C}(x):H_{x}\rightarrow H$, satisfying
\begin{enumerate}
\item [(1)] for any $x,y\in C$, $\rho_{1}(d(x,y))\leqslant\|t_{C}(x)(s(x))-t_{C}(y)(s(y))\|\leqslant \rho_{2}(d(x,y))$;
\item [(2)] for any subsets $C_{1},C_{2}\subset X\setminus K$ of diameter less than $r$ with $C_{1}\cap C_{2}\neq \emptyset$, there exists an affine isometry $t_{C_{1}C_{2}}: H\rightarrow H$ such that $t_{C_{1}}(x)\circ t^{-1}_{C_{2}}(x)=t_{C_{1}C_{2}}$ for all $x\in C_{1}\cap C_{2}$.
\end{enumerate}
\end{Def}

A metric space that admits a fibred coarse embedding into Hilbert space is abbreviated as an FCE space for simplicity. It was proved in \cite{CWY2013, GLWZ2024} that the (maximal) boundary coarse Baum-Connes conjecture holds for FCE spaces. As a result, the maximal coarse Baum-Connes conjecture and the coarse Novikov conjecture hold for FCE spaces. One can also replace the Hilbert space in Definition \ref{def: FCE} with any other metric space as a model space, including Hadamard manifolds, and $\ell^p$-spaces with $1\leq p <\infty$ as typical examples.

To deal with the relative coarse Baum-Connes conjecture, we shall refine the notion of fibred coarse embedding to the following version.

\begin{Def}\label{def: partial FCE}
Let $M$ be a metric space that serves as a model space (such as a Hilbert space, Hadamard manifolds, or \(\ell^p\)-spaces with \(1 \leq p < \infty\)). A metric space $X$ is said to admit a \emph{relative fibred coarse embedding into $M$ w.r.t. a subspace $Y\subseteq X$} if there exist\par
$\bullet$ a field of spaces $(M_x)_{x\in X}$ over $X$ such that each $M_x$ is isometric to $M$;\par
$\bullet$ a section $s:X\to\bigsqcup_{x\in X}M_x$, i.e. $s(x)\in M_x$ for each $x\in X$;\par
$\bullet$ two non-decreasing functions $\rho_-, \rho_+: \IR_+\to \IR_+$ with $\lim_{r\to\infty}\rho_{\pm}(r)=\infty$\\
such that, for any $R>0$, there exists $S>0$ such that there exists a trivialization
$$t_{x,R}:(M_z)_{z\in B(x,R)}\to B(x,R)\times M$$
for each $x\in X\backslash \text{Pen}(Y,S)$ satisfying the following conditions:
\begin{enumerate}
    \item [(1)] for any $z_1,z_2\in B(x,R)$, we have that
$$\rho_-(d(z_1,z_2))\leq d_M(t_{x,R}(z_1)(s(z_1)),t_{x,R}(z_2)(s(z_2)))\leq \rho_+(d(z_1,z_2));$$
\item [(2)] for any $x,y\in X\backslash \text{Pen}(Y,S)$ with $B(x,R)\cap B(y,R)\ne \emptyset$, there exists isometry $t_{xy,R}:M\to M$ such that $t_{x,R}(z)\circ t^{-1}_{y,R}(z)=t_{xy,R}$ for all $z\in B(x,R)\cap B(y,R)$.
\end{enumerate}
\end{Def}

Note that if the selected subset $Y$ is bounded, then a relative fibred coarse embedding of $X$ into $M$ w.r.t. $Y$ recovers to a fibred coarse embedding of $X$ into $M$. To get more intuition for the relative fibred coarse embedding structure, we give the following example.

Let $\{X_i\}_{i\in\{1,\cdots, N\}}$ be \emph{sparse} metric spaces, i.e., each $X_i=\bigsqcup_{n\in\IN}X_{i,n}$ is a coarse disjoint union of finite metric spaces. The product space
$$X=\prod_{i=1}^NX_i=\bigsqcup_{n_1,\cdots,n_N\in\IN}X_{1,n_1}\times\cdots\times X_{N,n_N},$$
endowed with the $\ell^2$-product metric, is actually a disjoint union of the countable sequence $(X_{1,n_1}\times\cdots\times X_{N,n_N})_{n_1,\cdots,n_N}$. For simplicity, we shall denote
$$X_{n_1,\cdots,n_N}=X_{1,n_1}\times\cdots\times X_{N,n_N}.$$
For each $i\in\{1,\cdots, N\}$ and $M\in\IN$, define
$$F^{(i)}_M=\bigcup_{n_i=1}^M\left(\bigcup_{n_1,\cdots,n_{i-1},n_{i+1},\cdots, n_N\in\IN}X_{n_1,\cdots,n_N}\right)$$
Such $F^{(i)}_M$ is called the \emph{$i$-th face with thickness $M$} of $X$. In particular, $F^{(i)}_1$ is called the \emph{$i$-th face} of $X$ for simplicity. Define $$F_M=\bigcup_{i=1}^NF^{(i)}_M,$$ which is called the $M$\emph{-boundary} of $X$. 

\begin{Pro}\label{pro: finite product of FCE spaces}
Let $(X_i)_{i=1}^N$ be a finite sequence of bounded geometric spaces that admits a fibred coarse embedding into Hilbert space. Assume that each $X_i$ is a coarse disjoint union of finite spaces as discussed above. Then the product space $X=\prod_{i=1}^NX_i$ admits a relative fibred coarse embedding into Hilbert space w.r.t. its $1$-boundary $F_1$.
\end{Pro}

First, we consider the case of the product of two FCE spaces. Assume that $X=(X_{n})_{n\in\IN}$ and $Y=(Y_{m})_{m\in\IN}$ are two sequences of metric spaces that admit a fibred coarse embedding into Hilbert space. Note that the product space $X\times Y$ does not necessarily admit a fibred coarse embedding into Hilbert space. It can be understood as follows. Let $\Ga$ be a residually finite group, and let $\{\Ga_n\}_{n\in\IN}$ be a \textit{filtration} of $\Ga$, which means that
$$\min\{|\gamma|\mid \gamma\ne e,\gamma\in \Ga_n\}\to\infty\text{ as }n\to\infty,$$
where $|\cdot|:\Ga\to\IR_+$ is a proper length function on $\Ga$. 
By \cite[Theorem 1.1]{CWW2013}, the box space $$\Box_{\{\Ga_n\}}\Ga=\bigsqcup_{n\in\IN}\Ga/\Ga_n$$ admits a fibred coarse embedding into Hilbert space if and only if $\Ga$ is a-T-menable. Now, consider $G$ to be another residually finite, a-T-menable group, and let $\{G_m\}_{m\in\IN}$ be a filtration of $G$. Then, the product of box spaces
$$\Box_{\{\Ga_n\}}\Ga\times \Box_{\{G_m\}}G=\bigsqcup_{n,m\in\IN}(\Ga\times G)/(\Ga_n\times G_m)$$
is a coarse disjoint union of finite quotient groups of $\Ga\times G$. It is straightforward to see that $\Ga\times G$ remains a-T-menable. However, the collection $\{\Ga_n\times G_m\}$ does not constitute a filtration of $\Ga\times G$, as there are infinitely many indices $(n,m)\in\IN^2$ for which $$\min\{|(\gamma,g)|\mid (\gamma,g)\ne e,\gamma\in \Ga_n, g\in G_m\}<R$$ for any $R>0$. This illustrates why the product of two fibred coarsely embeddable spaces may not permit a fibred coarse embedding into Hilbert space.

\begin{proof}[Proof of Proposition \ref{pro: finite product of FCE spaces}]
Without loss of generality, we assume that the family $(X_i)_{i=1}^N$ fibred coarsely embeds into the same Hilbert space $\H$ with the same controlling function $\rho_{\pm}$. We then have that
\begin{enumerate}
\item [(1)] a field of Hilbert space $(\H_{x_i})_{x_i\in X_i}$ over $X_i$ such that each $\H_{x_i}$ is isometric to $\H$;
\item [(2)] a section $s_i:X_i\to\bigsqcup_{x_i\in X_i}\H_{x_i}$, i.e. $s_i(x_i)\in \H_{x_i}$ for each $x_i\in X_i$. 
\end{enumerate}
Set $\wt\H=\bigoplus_{i=1}^N\H$. For each $x=(x_1,\cdots,x_N)\in X$, set
$$\wt\H_x=\bigoplus_{i=1}^N\H_{x_i},$$
and $s:X\to\bigsqcup_{x\in X}\wt\H_{x}$ to be
$$s(x)=(s_1(x_1),\cdots,s_N(x_N)).$$
For each $R>0$, there exists $M_i>0$ for each $i$ such that there exists a trivialization
$$t_{x_i,R}: (\H_{z_i})_{z_i\in B(x_i,R)}\to B(x_i,R)\times \H$$
for each $x_i\in \bigsqcup_{n\geq M_i}X_{i,n}$. Take $M=\max\{M_i\mid i=1,\cdots,N\}$, then the $M$-boundary $F_M$ forms a bounded neighborhood of $F_1$. For each $x=(x_1,\cdots,x_N)\in X\backslash F_M$, we define the trivialization
$$t_{x,R}: (\wt\H_{z})_{z\in \prod_{i=1}^NB(x_i,R)}\to \prod_{i=1}^NB(x_i,R)\times \wt\H$$
to be
$$(t_{x,R}(z))(v)=\left((t_{x_1,R}(z_1))(v_1),\cdots,t_{x_N,R}(z_N))(v_N)\right),$$
where $z=(z_1,\cdots,z_N)$, and $v=(v_1,\cdots,v_N)$. Since $B(x, R)$ is contained in $\prod_{i=1}^NB(x_i, R)$, we here replace the ball neighborhood with the square neighborhood. We then check that this trivialization satisfies the two conditions in Definition \ref{def: partial FCE}.

For any $z=(z_1,\cdots,z_N), z'=(z_1',\cdots,z_N')\in B(x,R)$, one can check that
\[\|t_{x,R}(z)(s(z))-t_{x,R}(z')(s(z'))\|^2=\sum_{i=1}^N\|t_{x_i,R}(z_i)(s_i(z_i))\|^2\leq N\cdot\rho^2_+(d(z,z')).\]
On the other hand, there exists at least one pair $(z_i,z_i')$ such that $d(z_i,z_i')\geq d(z,z')/N$. Thus we also have that
\[\sum_{i=1}^N\|t_{x_i,R}(z_i)(s_i(z_i))\|^2\geq \rho_-^2(d(z,z')/N).\]
Just take $\wt\rho_-(t)=\rho_-(t/N)$ and $\wt\rho_+(t)=\sqrt N\cdot \rho_+(t)$, and we have that
\[\wt\rho_-(d(z,z'))\leq \|t_{x,R}(z)(s(z))-t_{x,R}(z')(s(z'))\|\leq \wt\rho_+(d(z,z')).\]
This proves the first one.

For the second one, for any $x,y\in X\backslash F_M$ with $B(x,R)\cap B(y,R)\ne\emptyset$, we can check that $t_{x,R}(z)\circ t_{y,R}^{-1}(z)$ is equal to $t_{x_i,R}(z_i)\circ t_{y_i,R}^{-1}(z_i)=t_{x_iy_i,R}$ when we restrict it on the $i$-th direct sum component of $\wt\H=\bigoplus\H$. Therefore, 
\[t_{x,R}(z)\circ t_{y,R}^{-1}(z)=\bigoplus_{i=1}^Nt_{x_iy_i,R},\]
which is an isometry on $\wt\H$ independent on the choice of $z$. This finishes the proof.
\end{proof}

As illustrated in Figure \ref{fig: products of two FCE spaces}, $F_{1}$ corresponds precisely to the first row and the first column (the yellow region) and is referred to as the ``$1$-boundary" of $X$.

\begin{figure}[h]
\centering
\includegraphics[width=0.6\linewidth]{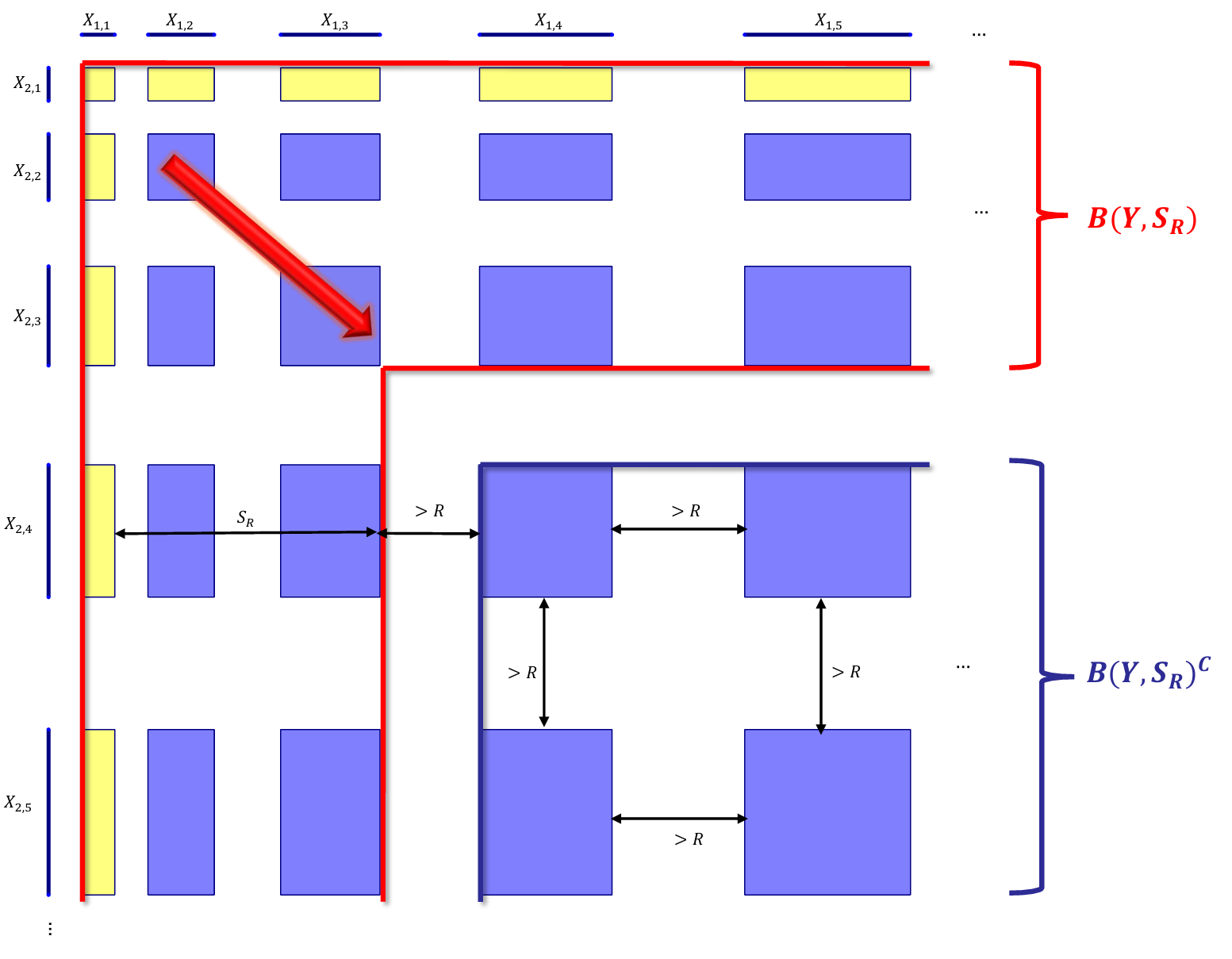}
\caption{Products of two FCE spaces}
\label{fig: products of two FCE spaces}
\end{figure}

\begin{Rem}
Note that the model space in Proposition \ref{pro: finite product of FCE spaces} can be substituted with a Hadamard manifold (i.e., a simply connected, complete, non-positively curved Riemannian manifold). The proof would remain unchanged.
Moreover, one can replace the model space with an $\ell^p$ space ($p\in[1,\infty)$). In this case, one should equip it with the $\ell^p$ product metric, which is coarsely equivalent to the $\ell^2$ product metric, given that there are only finitely many components in the product. Additionally, the model space should be the $\ell^p$ direct sum.

More generally, if we remove the assumption that $X_{1}$ and $X_{2}$ can each be expressed as a coarse disjoint union of finite metric spaces, the product space $X$ will admit a relative fibred coarse embedding into Hilbert space w.r.t.  $X_{1}\times\{x_2\}\cup X_{2}\times \{x_1\}$, where $x_1\in X_1$, $x_2\in X_2$ are based points. The core idea of the proof remains the same.
\end{Rem}

The following theorem is one of the main results in this paper. We shall use the remaining part of this section to prove it.

\begin{Thm}\label{thm: RCBC for RFCE spaces}
If $X$ admits a relative fibred coarse embedding into Hilbert space w.r.t. $Y\subseteq X$, then the relative coarse Baum-Connes conjecture for $(X,Y)$ and its maximal version hold, simultanesouly, i.e., the relative coarse assembly maps $\mu_{Y,\infty}$ and $\mu_{\max,Y,\infty}$ are both isomorphisms.

As a consequence, the canonical quotient map:
$$\pi_*: K_*(C^*_{\max,Y,\infty}(X))\to K_*(C^*_{Y,\infty}(X))$$
is an isomorphism.
\end{Thm}

To prove Theorem \ref{thm: RCBC for RFCE spaces}, we shall first reduce the theorem to a special case.

\begin{Def}\label{def: sparse with respect to Y}
Let $(X,d)$ be a metric space with bounded geometry and $Y$ be a subspace of $X$. Then $X$ is said to be \emph{sparse with respect to $Y$} if $X=\bigsqcup_{n\in\mathbb{N}} X_{n}$ and 
\begin{itemize}
\item[(1)] $Y$ is a a subset of $X_0$ and the canonical inclusion $Y\hookrightarrow X_0$ is a coarse equivalence;
\item[(2)] for each $n\in\mathbb{N}$, there exists $S_{n}>0$ such that $X_{n}\subseteq \text{Pen}(Y,S_{n})$;
\item[(3)]$d(X_{n},X_{m})\rightarrow\infty$ as $m+n\rightarrow\infty$ and $m\neq n$.
\end{itemize} 
\end{Def}

When $Y$ is bounded, then $X$ is sparse with respect to $Y$ is just a coarse disjoint union of a sequence of finite spaces, i.e., $X$ is a sparse space.

\begin{Exa}\label{exa: relative sparse}
Let $(X_i)_{i=1}^N$ be a finite sequence of sparse spaces. Then $\prod_{i=1}^NX_i$ is sparse with respect to its $1$-boundary. As Figure \ref{fig: products of two FCE spaces}, in the direction away from the 1-boundary along the diagonal (represented by the red arrows), the space becomes increasingly dispersed.
\end{Exa}

\begin{Pro}\label{pro: reduction to sparse spaces}
To prove Theorem \ref{thm: RCBC for RFCE spaces}, it suffices to prove it for $X$, which is sparse with respect to $Y$.
\end{Pro}

\begin{proof}
For any metric space $X$ with bounded geometry and $Y\subseteq X$, define
$$X_n=\{x\in M\mid n^3-n<d(x,Y)<(n+1)^3+n+1\}$$
for each $n\in\IN$. Then $X=\bigcup_{n\in\IN} X_n$. Denote by
$$ X^{(0)}=\bigsqcup_{n\in\mathbb{N}}X_{2n}\quad \text{and}\quad X^{(1)}=\bigsqcup_{n\in\mathbb{N}}X_{2n+1}.$$
Then $X^{(0)}$, $X^{(1)}$ and $X^{(0)}\cap X^{(1)}$ are all disjoint unions of subspaces of $X$. It is direct to check that $\{X^{(0)},X^{(1)}\}$ forms an $\omega$-excisive cover of $X$. Moreover, it is direct to check that the three spaces \(X^{(0)}\), \(X^{(1)}\), and \(X^{(0)} \cap X^{(1)}\) all contains the subset $Y$ and are all sparse with respect to $Y$. Since $X$ admits a relative fibred coarse embedding into Hilbert space w.r.t. $Y$, so does \(X^{(0)}\), \(X^{(1)}\), and \(X^{(0)} \cap X^{(1)}\). Then the proposition follows from a Mayer-Vietoris argument.
\end{proof}

\subsection{Coarsely proper algebras}

In this subsection, we will introduce a relative version of coarsely proper algebra that is associated with our relative fibred coarse embedding into Hilbert space, as defined in Section \ref{section 3.1}. Subsequently, several examples will be provided to help understand the concepts discussed.

By Proposition \ref{pro: reduction to sparse spaces}, we will always assume that the metric space $X=\bigsqcup_{n\in\IN}X_n$ is sparse with respect to the subspace $Y$ in the following. We define for each $k\in\IN$,
$$Y_k=\bigsqcup_{n=0}^{k}X_{n}.$$

\begin{Def}\label{definition of N-C-algebra}
Let $N$ be a second countable, locally compact, Hausdorff space. A $C^*$-algebra $\A$ is called a $N$-$C^*$-algebra if there exists a $*$-homomorphism $C_0(N)\to Z(M(\A))$ such that $C_0(N)\cdot \A$ is dense in $\A$, where $M(\A)$ is the multiplier algebra of $\A$ and $Z(M(\A))$ is the center of $M(\A)$.
\end{Def}

\begin{Def}[\cite{GLWZ2023}]\label{def: 1coarsely proper algebra}
A $C^*$-algebra $\A$ is said to be a \emph{coarsely proper algebra} associated with a coarse embedding $f: X\to M$ (comes with a pair of control functions $\rho_{\pm}$) if $\A$ is a $N$-$C^*$-algebra for a separable locally compact Hausdorff space $N$ such that
\begin{enumerate}
    \item[(1)]for any $R>0$, there exists a nest of open sets $O_{x,R}\subseteq N$ such that $\bigcup_{R>0} O_{x,R}$\ is dense in $N$ for each $x\in X$;
    \item[(2)] there exists a map $\varphi: M\to N$ such that  $\varphi(M)\cap O_{x,R}=\varphi(B_M(f(x),R))$;
    \item[(3)] $O_{x,\rho_-(R)}\cap O_{x', \rho_-(R)}=\emptyset$ whenever $d(f(x),f(x'))\geq 2R$;
    \item[(4)] $O_{x,R}\subseteq O_{x',R+\rho_+(d(x,x'))}$ for any $x,x'\in X$.
\end{enumerate}
\end{Def}

\begin{Rem} The requirement that the open subset satisfies $O_{x_{1},r}\subseteq O_{x_{2},\rho_{+}(d(x_{1},x_{2}))+r}$ rather than
$O_{x_{1},r}\subseteq O_{x_{2},d(x_{1},x_{2})+r}$ for any $x_{1},x_{2}\in X$ and $r>0$ in Definition \ref{def: 1coarsely proper algebra} arises from the fact that all open subsets $\{O_{x,r}\}_{x\in X}$ are chosen within the metric space $M$ , which possesses a distinct metric compared to that of $X$. So we need a control function $\rho_{+}$ to ensure consistency with the coarse embedding. We also note that this requirement guarantees the open sets  $\{O_{x,r}\}_{x\in X}$ we identify do not shrink excessively. The reason for asking $O_{x_{1},\ \rho_{-}(r)}\cap O_{x_{2},\ \rho_{-}(r)}=\emptyset$ if $d(x_{1},x_{2})>2r$ for any $x_{1},x_{2}\in X$ instead of $O_{x_{1},\rho_-(r)}\cap O_{x_{2},\rho_{-}(r)}=\emptyset$ is quite similar. 
\end{Rem}

As defined in Definition \ref{def: partial FCE}, when approaching infinity along pre-selected directions away from $Y$, for each $k$ and $x\in Y_k^c$, there exists a trivialization
$$t_x: (M_{z})_{z\in B(x,l_k)}\to M\times B(x,l_k).$$
The map $f_{x,k}: z\mapsto t_x(z)(s(z))$ actually gives a coarse embedding of $B(x,l_k)$ into $M$ with controlling functions $\rho_{\pm}$. If $B(x,l_k)\cap B(y,l_k)\ne\emptyset$, then there exists a isometry $t_{xy}: M\to M$ such that $t_{xy}=t_x(z)\circ t^{-1}_y(z)$ for any $z$ in the intersection part.

\begin{Def}\label{def: coarsely proper algebra}
A $N$-$C^*$-algebra $\A$ is called a \emph{coarsely proper algebra} associated with the relative fibred coarse embedding for $X$ into $M$ w.r.t. $Y$ if there exists\begin{itemize}
\item a field of $C^*$-algebras $(\A_x)_{x\in X}$ over $X$ such that each $\A_x$ is isomorphic to $\A$;
\item a nest of open sets $\{O_{x,R}\}_{R\geq 0}$ in $N_x$ such that $\bigcup_{R\geq 0}O_{x,R}$ is dense in $N_x$ for any $x\in X$, where $\A_x$ is a $N_x$-$C^*$-algebra.
\end{itemize}
such that for any $k\in\IN$ and $x\in X\backslash Y_k$, the trivialization for fibred coarse embedding induces a trivialization for the field of $C^*$-algebras (It is important to note that, although there is a slight abuse of notation, we will continue to refer to this trivialization as \( t \).)
$$t_x:\bigcup_{z\in B(x,l_k)}\A_z\to B(x,l_k)\times\A$$
such that
\begin{enumerate}
\item [(1)] The mapping $t_x(z):\A_z\to\A$ is a $C^*$-isomorphism that induces a homeomorphism $f_{t_x(z)}: N_x\to N$ such that there exists a pair of non-decreasing proper functions $\rho_{\pm}$ from $\IR_+$ to $\IR_+$ such that $f_{t_x(z)}(O_{z,\rho_-(R)})\cap f_{t_x(z')}(O_{z', \rho_-(R)})=\emptyset$ whenever $$d(t_x(z)(s(z)),t_x(z')(s(z')))\geq 2R,$$ and $f_{t_x(z)}(O_{z,R})\subseteq f_{t_x(z')}(O_{z', {R+\rho_{+}(d(z,z'))}})$ for any $z,z'\in B(x,l_k)$;
\item [(2)] for any $x,y\subset X\backslash Y_k$ with $B(x,R)\cap B(y,R)\ne \emptyset$, there exists isomorphism $t_{xy}:\A\to \A$ such that $t_{x,R}(z)\circ t^{-1}_{y,R}(z)=t_{xy}$ for all $z\in B(x,R)\cap B(y,R)$.
\end{enumerate}\end{Def}

\begin{Exa}\label{exa: coarsely proper algebra}
If $X$ admits a relative fibred coarse embedding into an $\ell^p$-space $B$ w.r.t. a fixed subspace $Y$, then the algebra $\A(B)$, as defined in \cite{GLWZ2024}, satisfies Definition \ref{def: coarsely proper algebra}. We will not restate the definition here and directly use the notation established in \cite{GLWZ2024}. It is clear that $\A(B)$ is a $B\times\IR_+$-algebra. Since $X$ admits a relative fibred coarse embedding into $B=\ell^p(\IN,\IR)$ w.r.t. $Y$, we assign $\A(B_x)$ to each $x$. For any $R>0$ and $x\in X$, define
$$O_{x,R}=\{(v,t)\in B_x\times\IR_+\mid \|v-s(x)\|^p_p+t^2< R^2\}\subseteq B_x\times\IR_+.$$
For the isometry $t_x(z): B_x\to B$, it induces an isomorphism $\A(B_x)\to\A(B)$ by mapping the Clifford generator to the Clifford generator, i.e.
$$C_{B_x,v}\mapsto C_{B,t_x(z)(v)}.$$
Then the family $\{t_{x}(z)(O_{z,R})\}_{z\in B(x,l_k), R\geq 0}$ makes $\A(B)$ a coarsely proper algebra for the coarse embedding $B(x,l_k)\to B$ defined by $z\mapsto t_x(z)(s(z))$. This proves that $\A(B)$ is a coarsely proper algebra associated with the relative fibred coarse embedding.
Especially if we take $p=2$, then $\A(B)$ is isomorphic to $\lim_{V_a\subseteq\H}\mathcal{S}\C(V_a)$, as defined in \cite{HKT1998, Yu2000}. Analogous to the explanation above, a proper algebra $\mathcal{A}(M)$ can be similarly constructed for relative fibred coarse embeddings into a Hadamard manifold $M$, as detailed in \cite{SW2007}.
\end{Exa}

Combining \cite[Definition 3.7]{GLWZ2023} and \cite[Definition 3.4]{DGWY2025}, one can similarly define the \emph{(maximal) twisted Roe algebra at infinity of $X$ relative to $Y$}, $C^*_{Y,\infty}(P_d(X),\A)$ and the \emph{(maximal) twisted localization algebra at infinity of $X$ relative to $Y$}, $C^*_{L,Y,\infty}(P_d(X),\A)$ with coefficient in a coarsely proper algebra $\A$ for a metric spaces $X$ \emph{which is sparse with respect to $Y$}. Since the constructions are quite similar, we shall omit the details here.

By the universal property of the maximal norm, there is a canonical quotient map
\[\pi: \Cau(P_d(X),\A)\to C^*_{Y,\infty}(P_d(X),\A).\]
Moreover, the evaluation maps $g\mapsto g(0)$ for both maximal and reduced cases induce the following \emph{twisted relative assembly maps}: 
$$\mu_{\max,\A}: \lim_{d\to\infty}K_*(\CauL(P_d(X),\A))\to \lim_{d\to\infty}K_*(\Cau(P_d(X),\A));$$
$$\mu_{\A}: \lim_{d\to\infty}K_*(C^*_{L,Y,\infty}(P_d(X),\A))\to \lim_{d\to\infty}K_*(C^*_{Y,\infty}(P_d(X),\A)).$$

\begin{Thm}\label{thm: twisted partial assembly map}
If $\A$ is a {coarsely proper algebra} associated with the relative fibred coarse embedding of $X$ w.r.t. $Y$, then the twisted relative assembly map $\mu_{\max,\A}$ and $\mu_{\A}$ are isomorphisms.\qed
\end{Thm}



\subsection{A conceptual description of uniformly almost flat Bott generators}

In this subsection, we provide a conceptual description of \emph{uniformly flat Bott generators} for coarsely proper algebras and present the proof of Theorem~\ref{thm: RCBC for RFCE spaces}.

\begin{Def}
A coarsely proper algebra $\A$ is said to admit a family of \emph{uniformly almost flat Bott generators} if,\begin{itemize}
\item for any $x\in X$ and $\t\geq 1$, there exists a generator $[b_{x,\t}]\in K_*(\A_x)$ such that for any $\varepsilon>0$ and $r\geq 0$, there exists $T>1$ such that
$$\|t_x(x)(b_{x,\t})-t_x(y)(b_{y,\t})\|\leq \varepsilon$$
for all $x,y$ with $d(x,y)\leq r$ and $x,y$ in the same trivialization, and $\t\geq T$, where $\A_x$ and $t_x$ is defined as in Definition \ref{def: coarsely proper algebra};
\item for any $\t\geq 1$, there exists $R>0$ such that $\supp(b_{x,\t})\subseteq O_{x,R}$ for any $x\in X$, where $O_{x,R}$ is defined as in Definition \ref{def: coarsely proper algebra}.
\end{itemize}\end{Def}

A typical example of a uniformly flat Bott generator in the $K_0$-group comes from Hadamard manifolds; see the discussion of Bott generators in \cite[Section 4]{SW2007}. For the $K_1$-case, we illustrate the notion of a uniformly flat Bott generator with the following example.

\begin{Exa}[$\ell^p$-spaces]\label{exa: Bott generator for lp}
Let $B_x=\ell^p(\IN,\IR)$ for each $x\in X$ and $H=\ell^2(\IN,\IR)$. For any $\t\in\IR_+$ and $x\in X$, denote $\beta_{x,\t}:\S\to\A(B_x)\ox\K$ by the map
$$\beta_{x,\t}(f)=f_\t(C_{s(x)})\ox p,$$
where $C: B_x\to\Cl(H)$ is the Clifford generator defined as in \cite[Section 3]{GLWZ2024}  by $p/2$-H\"{o}lder extension of the Mazur map, and $f_\t(r)=f(r/\t)$. By \cite[Remark 2.9.13]{WY2020}, each element $\beta_{x,\t}$ determines an element in $[C_0(\IR),\A(B_x)\ox\K]$, thus determines an element in $K_1(\A(B_x))$. Moreover, it is proved in \cite[Theorem 3.15]{GLWZ2024} that $\beta_{x,\t}$ induces an isomorphism $(\beta_{x,\t})_*: K_*(\S)\to K_*(\A(B_x))$. Thus, $[\beta_{x,\t}]$ is a generator of $K_1(\A(B_x))$ for each $x\in X$ and $\t\in\IR_+$.

We now prove that $\{\beta_{x,\t}\}_{x\in X,\t\in\IR_+}$ is a family of uniformly almost flat Bott generators. Since the $K_1$-class determined by $\beta_{x,\t}$ is only dependent on its value at the generator $\frac{1}{x\pm i}$ of $\S$, it suffices to prove for fixed $f\in \S$ and any $\varepsilon>0$, $R>0$, there exists $T>1$ such that
\[\|t_x(x)(\beta_{x,\t}(f))-t_x(y)(\beta_{y,\t}(f))\|\leq\varepsilon\]
for all $x,y\in X$ with $d(x,y)\leq R$ and $x,y$ in a the same trivialization, and $\t\geq T$. This has been proved in \cite[Lemma 3.8]{GLWZ2024}.
\end{Exa}

A sequence of almost flat Bott generators induces Bott maps
$$\beta: K_*(C^*_{Y,\infty}(P_d(X))) \to K_{*}(C^*_{Y,\infty}(P_d(X),\A)).$$
If the Bott generator belongs to $K_1(\mathcal{A})$, as shown in Example \ref{exa: Bott generator for lp}, then a method analogous to \cite[Definition 3.14]{DGWY2025} can be applied to construct $\beta$. If, instead, the Bott generator lies in $K_0(\mathcal{A})$, the construction of $\beta$ follows from \cite[Section 4]{SW2007}. Similar considerations apply to the localization algebra setting. Since the constructions are largely similar, we omit the details here.

For both cases, there always exists the following commuting diagram
\begin{equation}\label{eq: dual-Dirac}\begin{tikzcd}
K_{*}(C^*_{L, Y,\infty}(P_d(X))) \arrow[d,"\beta_L"'] \arrow[r,"ev_*"] & K_*(C^*_{Y,\infty}(P_d(X))) \arrow[d,"\beta"] \\
K_*(C^*_{L, Y,\infty}(P_d(X),\A)) \arrow[r,"ev_*"]           & K_*(C^*_{Y,\infty}(P_d(X),\A))          
\end{tikzcd}\end{equation}
The map $\beta_L$ and $\beta$ will change the degree on $K$-theory if the family of uniformly almost flat Bott generators we choose is in $K_1(\A)$.

\begin{Lem}\label{lem: isomorphism for Bott localization}
The Bott map for localization algebras associated with a family of uniformly almost flat Bott generators is an isomorphism.
\end{Lem}

\begin{proof}
By using a cutting and pasting argument as in \cite[Theorem 7.4]{GLWZ2024} for the localization algebra, it suffices to prove the Bott map for the zero-skeleton $\Delta^{(0)}$ of the Rips complex. In this case, we have that
$$K_*(C^*_{L,Y,\infty}(\Delta^{(0)}))\cong\frac{\prod_{x\in X}K_*(C^*_L(\{x\}))}{\prod_{y\in Y}K_*(C^*_L(\{y\}))}\quad\text{and}\quad K_*(C^*_{L,Y,\infty}(\Delta^{(0)},\A))\cong\frac{\prod_{x\in X}K_*(C^*_L(\{x\},\A))}{\prod_{y\in Y}K_*(C^*_L(\{y\},\A))}.$$
The map from $K_*(C^*_L(\{x\}))$ to $K_*(C^*_L(\{x\},\A))$ is given by the Bott map which is essentially an exterior tensor product with the Bott generator, thus is an isomorphism.
\end{proof}

Now, we are ready to prove Theorem \ref{thm: RCBC for RFCE spaces}.

\begin{proof}[Proof of Theorem \ref{thm: RCBC for RFCE spaces}]
One can define the inverse map of the Bott map as in \cite[Section 7.2]{CWY2013}, which leads to the following commuting diagram:
\[\begin{tikzcd}
K_{*}(C^*_{L, Y,\infty}(P_d(X))) \arrow[d,"\beta_L"'] \arrow[r,"ev_*"] & K_*(C^*_{Y,\infty}((P_d(X))) \arrow[d,"\beta"] \\
K_{*}(C^*_{L, Y,\infty}(P_d(X),\A)) \arrow[d,"\alpha_L"'] \arrow[r,"ev_*","\cong"'] & K_*(C^*_{Y,\infty}(P_d(X),\A)) \arrow[d,"\alpha"] \\
K_*(C^*_{L, Y,\infty}(P_d(X))) \arrow[r,"ev_*"]           & K_*(C^*_{Y,\infty}(P_d(X))).          
\end{tikzcd}\]
One can similarly check that $\alpha\circ\beta$ and $\alpha_L\circ\beta_L$ are both identity maps. Then Theorem \ref{thm: RCBC for RFCE spaces} follows from a diagram chasing argument.
\end{proof}

As a direct corollary, we have the following result:

\begin{Cor}\label{cor: RCNC for FCE into lp and Hadamard}
Let $X$ be a metric space that is sparse relative to $Y$. If $X$ admits a fibred coarse embedding into an $\ell^p$-space (with $p\in[1,\infty)$) or a Hadamard manifold, then the relative coarse Novikov conjecture for $(X,Y)$ holds.
\end{Cor}

\begin{proof}
It is a direct corollary of Theorem \ref{thm: twisted partial assembly map}, Example \ref{exa: coarsely proper algebra}, Lemma \ref{lem: isomorphism for Bott localization}, Example \ref{exa: Bott generator for lp} and a diagram chasing argument on \eqref{eq: dual-Dirac}.
\end{proof}




\section{Applications: from ``relative'' to ``global''}

In this section, we present several applications of the relative coarse Baum-Connes conjecture. We first prove that under certain conditions, the relative conjecture can be used to establish the global coarse Baum-Connes conjecture. In fact, the relationship between the coarse Baum-Connes conjecture for a product space $X\times Y$ and the conjectures for the individual factors $X$ and $Y$ is not well understood in general; a precise relationship is only known in the twisted (i.e., coefficient) case \cite{DG2024}. As a concrete application, we prove that the maximal coarse Baum-Connes conjecture holds for products of finitely many FCE (fibred coarsely embeddable) spaces. This result provides explicit examples of spaces where the maximal coarse Baum-Connes conjecture holds while fibred coarse embeddability (FCE) may fail.

\subsection{From relative to global}

The following theorem is the main result of this section.

\begin{Thm}\label{thm: reduction to infinity for maximal}
Let $X$ be a sparse metric space relative to a subspace $Y\subseteq X$.\begin{itemize}
\item[(1)] If $Y$ satisfies the maximal coarse Baum-Connes conjecture. Then the maximal relative coarse Baum-Connes conjecture for $(X,Y)$ is equivalent to the maximal coarse Baum-Connes conjecture for $X$.
\item[(2)] If the relative coarse Novikov conjecture holds for $(X,Y)$ and the coarse Novikov conjecture holds for $Y$, then the coarse Novikov conjecture holds for $X$.
\end{itemize}\end{Thm}

Before we can prove Theorem \ref{thm: reduction to infinity for maximal}, we shall need the following Lemma.

\begin{Lem}\label{lem: maximal Roe exact sequence}
For $X$ sparse relative to $Y$, we have the following exact sequence on $K$-theory
$$0\to K_*(C^*_{\max}(X,Y))\to K_*(C^*_{\max}(X))\to K_*(C^*_{\max, Y,\infty}(X))\to 0.$$
\end{Lem}

\begin{proof}
We provide two approaches to this lemma here.

\noindent{\bf Approach 1.} For the first one, the idea of the proof is similar to \cite[Proposition 2.10]{OY2009}. It is proved in \cite{HRY1993} that $C^*_{\max}(X, Y)=\lim\limits_{R\to\infty}C^*_{\max}(\text{Pen}(Y, R))=\lim\limits_{k\to\infty}C^*_{\max}(Y_k)$. Since $Y_k$ is coarsely equivalent with $Y$ for any $k$, we conclude that $K_*(C^*_{\max}(X, Y))\cong K_*(C^*_{\max}(Y))$. Moreover, the canonical inclusion map $C^*(Y)\to C^*(X, Y)$ induces this isomorphism. Thus, by the six-term exact sequence in $K$-theory, it suffices to prove that the inclusion $i: C^*(Y)\to C^*(X)$ induces an injection on $K$-theory.

For any unitary $u\in C^*_{\max}(Y)^+$, if $i_*([u])=0\in K_1(C^*_{\max}(X))$, then there exists a homotopy of unitary $u_t\in C^*_{\max}(X)^+$ such that $u_0=u$ and $u_1=I$. Here we do not need to tensor both sides with matrix algebra when taking this homotopy since $C^*_{\max}(X)$ is quasi-stable. Take a sufficiently small $\varepsilon$, we can take a homotopy $w_t\in\IC[X]$ such that $\|w_t-u_t\|\leq\varepsilon/3$ and $\prop(w_t)\leq r$ for all $t\in [0,1]$. Then, $w_t$ is a $\varepsilon$-unitary for each $t\in[0,1]$. Since $w_t$ has propagation less than $r$, there exists $k>0$ such that $w_t$ is split into a direct sum $w'_t\in\IC[Y_k]$ and $w''_t\in\IC[X\backslash Y_k]$ by ($2_b$) in Definition \ref{def: sparse with respect to Y}. Then $w'_t$ forms a homotopy of $\varepsilon$-unitary. Take a polar decomposition to $w'_t$ pointwise, we then have $w_t=a_tv_t$ where $v_t$ is a continuous path of unitaries in $C^*_{\max}(Y_k)$. Since $v_0$ is close to $u$ and $v_1$ is close to $I$, thus there are continuous paths linking $u$ and $v_0$, $v_1$ and $I$ in $C^*_{\max}(Y_k)^+$, respectively (see for example \cite[Lemma 2.1.3]{RLL2000}). Combining these two paths and $v_t$, we then get a continuous path of unitaries connecting $u$ and $I$ in $C^*_{\max}(X,Y)$. This shows that
$$K_1(C^*(Y))\xrightarrow{\cong} K_1(C^*(X, Y))\to K_1(C^*(X))$$
is an injection. The analog proof also holds for $K_0$, see \cite[Proposition 2.10]{OY2009} for details for $K_0$. This finishes the proof.

\noindent{\bf Approach 2.}  We can also prove this lemma by showing the quotient map
$$q: C^*_{\max}(X)\to C^*_{\max, Y,\infty}(X)$$
induces a surjection on $K$-theory. By \cite[Lemma 12.5.4]{WY2020}, the maximal coarse Baum-Connes conjecture of the coarse disjoint union of $(X_{n})_{n\in\IN}$ is equivalent to the separated disjoint union of $(X_{n})_{n\in\IN}$. By separated disjoint union, we mean that the metric between two different components is infinite, i.e., for any $x\in X_{n}$ and $x'\in X_{n'}$,
\[d\left(x,x'\right)=\left\{\begin{aligned}&\infty,&&\text{if }n\ne n'\\&d_{X_n}(x,x'),&&\text{otherwise.}\end{aligned}\right.\]
Define $X'$ to be the separated disjoint union of $(X_{n})_{n\in\IN}$. Then the algebraic Roe algebra of $X'$ is a subalgebra of $\prod_{n\in\IN}\IC[X_{n}]$ consisting of all sequences of operators with uniformly finite propagations. Denoted by $Y'_k$ the same set as $Y_k$ but equipped with the separated disjoint union metric. We then have the following commuting diagram
\[\begin{tikzcd}
0 \arrow[r] & \lim\limits_{k\to\infty}K_*(C^*_{\max}(Y_k')) \arrow[r] \arrow[d] & K_*(C^*_{\max}(X')) \arrow[r, "p_*"] \arrow[d] &  K_*(C^*_{\max, Y,\infty}(X)) \arrow[r] \arrow[d, "\cong"', "id_*"]  &0\\
 & K_*(C^*_{\max}(X,Y)) \arrow[r]           & K_*(C^*_{\max}(X)) \arrow[r, "q_*"]      & K_*(C^*_{\max, Y,\infty}(X))                   & 
\end{tikzcd}\]
Here are some explanations for the diagram above. One can check that there is also an exact sequence of $*$-algebra
\[0\to\lim_{k\to\infty}\IC[Y_k']\to\IC[X']\to\ICu[X]\to 0,\]
which induces an isomorphism after taking the maximal completion
\[0\to\lim_{k\to\infty}C^*_{\max}(Y_k')\to C^*_{\max}(X')\to\Cau(X)\to 0.\]
The reader is also referred to \cite[Lemma 12.5.4]{WY2020} for relevant discussion. It is easy to check $C^*_{\max}(Y_k')\to C^*_{\max}(X')$ induces an injection on the level of $K$-theory since $\IC[X']\to\IC[Y_k']$ defined by
$$T\mapsto \chi_{Y_k'}T\chi_{Y_k'},$$
where $\chi_{Y_k'}$ is the characteristic function on $Y_k'$, extends to a map $C^*_{\max}(X')\to C^*_{\max}(Y_k')$ after taking $C^*$-completion under the maximal norm. This map is a right inverse to the canonical inclusion $C^*_{\max}(Y_k')\to C^*_{\max}(X')$. This means that $K_*(C^*_{\max}(Y_k'))\to K_*(C^*_{\max}(X'))$ is injective for each $k$. Passing $k$ to infinity, we then have that $\lim_{k\to\infty}K_*(C^*_{\max}(Y_k'))\to K_*(C^*_{\max}(X'))$ is injective. The upper horizontal sequence is induced by the above exact sequence of $C^*$-algebras. All the vertical maps are induced by the inclusion. Since $id_*\circ p_*$ is surjective, by diagram chasing, $q_*$ must be surjective. This finishes the proof.
\end{proof}

\begin{Rem}\label{rem: injection of geometric ideal}
We should mention here $K_*(C^*(X, Y))\to K_*(C^*(X))$ is still an injection when taking the reduced norm, and the proof is the same with the first approach we given in Lemma \ref{lem: maximal Roe exact sequence}.
\end{Rem}

One can similarly prove the localization algebra counterpart of Lemma \ref{lem: maximal Roe exact sequence}. The proof shares the same idea with the second approach as above, we leave the details to the readers.

\begin{Lem}\label{lem: maximal localization exact sequence}
With the notation as above, there is an exact sequence
$$0\to \lim\limits_{d\to\infty}\lim\limits_{k\to\infty}K_*(P_d(Y_k))\to\lim\limits_{d\to\infty}K_*(P_d(X))\to \lim\limits_{d\to\infty}K_*(\CauL(P_d(X)))\to 0.\qed$$
\end{Lem}

Now, we are ready to prove Theorem \ref{thm: reduction to infinity for maximal}.

\begin{proof}[Proof of Theorem \ref{thm: reduction to infinity for maximal} (1)]
By Lemma \ref{lem: maximal Roe exact sequence} and Lemma \ref{lem: maximal localization exact sequence}, we have the following commuting diagram on $K$-theory:
\[\begin{tikzcd}
0 \arrow[d]                    & 0 \arrow[d]            \\
\lim\limits_{d\to\infty}\lim\limits_{k\to\infty}K_*(P_d(Y_k)) \arrow[d] \arrow[r, "\mu_{\max,Y}"] & K_*(C^*_{\max}(X,Y)) \arrow[d, "\iota_*"] \\
\lim\limits_{d\to\infty}K_*(P_d(X)) \arrow[d] \arrow[r,"\mu_{\max,X}"]          & K_*(C^*_{\max}(X)) \arrow[d]            \\
\lim\limits_{d\to\infty}K_*(\CauL(P_d(X)))\arrow[d] \arrow[r,"\mu_{\max,Y,\infty}"]          &\lim\limits_{d\to\infty} K_*(\Cau(P_d(X))) \arrow[d]            \\
0                              & 0                     
\end{tikzcd}\]
All the horizontal maps are induced by the evaluation maps. To prove Theorem \ref{thm: reduction to infinity for maximal}, it suffices to prove the first horizontal map $\mu_{Y}$ is an isomorphism. We claim that $\mu_{Y}$ is equivalent to the maximal coarse assembly map for $Y$. Indeed, for $k\leq k'$ and $d\leq d'$, the canonical inclusion maps satisfy the following commuting diagram
\[\begin{tikzcd}
P_d(Y_k) \arrow[d] \arrow[r] &  P_d(Y_{k'})\arrow[d] \\
P_{d'}(Y_k) \arrow[r]           & P_{d'}(Y_{k'}) .         
\end{tikzcd}\]
The order of $\lim_{d\to\infty}$ and $\lim_{k\to\infty}$ on the first row can be exchanged. Thus, the $\mu_{Y}$ can be rewritten as
$$\mu_{Y}: \lim_{k\to\infty}\lim_{d\to\infty}K_*(P_d(Y_k))\to \lim_{k\to\infty}K_*(C^*_{\max}(Y_k))$$
which is the inductive limit of the assembly map of $Y_k$ as $k$ passes to infinity. Since $Y_k$ is coarsely equivalent to $Y$ and the (maximal) coarse Baum-Connes conjecture is invariant under coarse equivalence. Thus, to prove $\mu_{Y}$ is an isomorphism, it suffices to prove the maximal coarse Baum-Connes conjecture for $Y$. This finishes the proof.
\end{proof}

For the reduced case, Lemma \ref{lem: maximal Roe exact sequence} does not hold in this case. However, we have the following sequence
$$0\to K_*(C^*(X,Y))\xrightarrow{i_*} K_*(C^*(X))\xrightarrow{\pi_*} K_*(C^*_{u, Y,\infty}((X_{n})_{n\in\IN})).$$
Note that $\pi_*\circ i_*=0$, while this sequence is not exact at $K_*(C^*(X))$. The injectivity of $i_*$ is guaranteed by Lemma \ref{lem: maximal Roe exact sequence} and Remark \ref{rem: injection of geometric ideal}

\begin{proof}[Proof of Theorem \ref{thm: reduction to infinity for maximal} (2)]
Consider the following commuting diagram:
\[\begin{tikzcd}
0 \arrow[d]                    & 0 \arrow[d]            \\
\lim\limits_{d\to\infty}\lim\limits_{k\to\infty}K_*(P_d(Y_k)) \arrow[d] \arrow[r, "\mu_{Y}"] & K_*(C^*(X,Y)) \arrow[d, "\iota_*"] \\
\lim\limits_{d\to\infty}K_*(P_d(X)) \arrow[d] \arrow[r,"\mu_X"]          & K_*(C^*(X)) \arrow[d]            \\
\lim\limits_{d\to\infty}K_*(C^*_{L,Y,\infty}(P_d(X)))\arrow[d] \arrow[r,"\mu_{Y,\infty}"]          &\lim\limits_{d\to\infty} K_*(C^*_{Y,\infty}(P_d(X)))           \\
0                              &                      
\end{tikzcd}\]
As we discussed in the part (1) of Theorem \ref{thm: reduction to infinity for maximal}, the map $\mu_Y$ on the upper row is equivalent to the reduced assembly map for $Y$. Then the theorem follows by using a diagram chasing.
\end{proof}

As a corollary of Theorem \ref{thm: RCBC for RFCE spaces}, Corollary \ref{cor: RCNC for FCE into lp and Hadamard} and Theorem \ref{thm: reduction to infinity for maximal}, we have the following result.

\begin{Thm}\label{thm: CBC for RFCE}
Let $X$ be a metric space with bounded geometry which is sparse with respect to $Y\subseteq X$. \begin{itemize}
\item[(1)] If $Y$ satisfies the maximal coarse Baum-Connes conjecture and $X$ admits a relative fibred coarse embedding into Hilbert space w.r.t. $Y$, then the maximal coarse Baum-Connes conjecture holds for $X$.
\item[(2)] If $Y$ satisfies the coarse Novikov conjecture and $X$ admits a relative fibred coarse embedding into an $\ell^p$-space (with $p\in[1,\infty)$) or a Hadamard manifold w.r.t. $Y$, then the coarse Novikov conjecture holds for $X$.
\end{itemize}\end{Thm}

\subsection{Product of FCE spaces}

As a byproduct of Theorem \ref{thm: CBC for RFCE}, we have the following corollary:

\begin{Cor}\label{cor: FCEH times FCEH}
Let $(X_i)_{i=1}^N$ be a finite family of bounded geometry metric spaces that admit a fibred coarse embedding into Hilbert space. Then the maximal coarse Baum-Connes conjecture holds for $\prod_{i=1}^NX_i$.
\end{Cor}

As we noted earlier, the product of finitely many fibred coarsely embeddable (FCE) spaces may not be FCE, and it's unclear if the coarse Baum-Connes conjecture holds for such products (only the twisted case has been studied \cite{DG2024}). Therefore, these results actually provide many examples of spaces that satisfy the maximal coarse Baum-Connes conjecture but are not FCE.

Before we can prove Corollary \ref{cor: FCEH times FCEH}, we need the following lemma.

\begin{Lem}\label{lem: reduction to coarse disjoint union}
To prove the maximal coarse Baum-Connes conjecture for $\prod_{i=1}^NX_i$, it suffices to prove it for the case where each $X_i$ is a sparse metric space, where $i=1,\cdots, N$.
\end{Lem}

\begin{proof}
We shall prove this lemma by induction. It is proved in \cite{CWY2013} that this lemma holds when $N=1$. Fix the base point $x_i\in X_i$ for each $i$, define
$$X_{i,n}=\left\{x\in X\mid n^3-n\leq d(x,x_i)\leq (n+1)^3+(n+1)\right\}.$$
Set
$$X^{(0)}_i=\bigcup_{n:\text{even}} X_{i,n}\quad\text{and}\quad X^{(1)}_i=\bigcup_{n:\text{odd}} X_{i,n}.$$
Since $X_i$ has bounded geometry, $X_i^{(0)}, X_i^{(1)}, X_i^{(0)}\cap X_i^{(1)}$ are all coarse disjoint unions of finite spaces and the pair $(X_i^{(0)}, X_i^{(1)})$ forms an $\omega$-excisive cover of $X_i$. Moreover, $X_i^{(0)}$, $X_i^{(1)}$ and $X_i^{(0)}\cap X_i^{(1)}$ all have FCE-structures for each $i$. By using a Mayer-Vietoris argument, to show the maximal coarse Baum-Connes conjecture for $\prod_{i=1}^NX_i$, it suffices to prove the maximal coarse Baum-Connes conjecture for
$$X_1^{(0)}\times\prod_{i=2}^NX_i,\quad X_1^{(1)}\times\prod_{i=2}^NX_i,\quad\text{and}\quad \left(X_1^{(0)}\cap X_1^{(1)}\right)\times\prod_{i=2}^NX_i.$$
By further induction, to show the maximal coarse Baum-Connes conjecture for $X_1^{(0)}\times\prod_{i=2}^NX_i$, it suffices to prove the maximal coarse Baum-Connes conjecture for
$$X_1^{(0)}\times X_2^{(0)}\times\prod_{i=3}^NX_n,\quad X_1^{(0)}\times X_2^{(1)}\times\prod_{i=2}^NX_i,\quad\text{and}\quad X_1^{(0)}\times \left(X_2^{(0)}\cap X_2^{(1)}\right)\times\prod_{i=2}^NX_i.$$
To sum up, it suffices to prove the maximal coarse Baum-Connes conjecture for the case that each $X_i$ is a coarse disjoint union of finite spaces.
\end{proof}

\begin{proof}[Proof of Corollary \ref{cor: FCEH times FCEH}]
As discussed in Lemma \ref{lem: reduction to coarse disjoint union}, it is sufficient to prove the case where each \(X_i\) is a coarse disjoint union of finite spaces. According to Theorem \ref{thm: CBC for RFCE}, Proposition \ref{pro: finite product of FCE spaces}, and Example \ref{exa: relative sparse}, it is sufficient to prove that the maximal coarse Baum-Connes conjecture holds for \(F_1\).

For $k=\{0,\cdots,N-1\}$, we define the \emph{$k$-skeleton of $X$}, denoted by $\Delta^{(k)}$, to be
\[\Delta^{(k)}=\bigcup_{1\leq i_1<\cdots<i_{N-k}\leq N}\left(F^{(i_1)}_1\cap\cdots\cap F^{(i_{N-k})}_1\right).\]
We take the product of three sparse spaces as an example, as shown in Figure \ref{Products of three FCE spaces}, where the elements in the product space resemble discrete cuboids. The three red lines represent three different spaces: $X_{1}$, $X_{2}$, and $X_{3}$. In this case, the 2-skeleton $\Delta^{(2)}$ represents the region on the three faces, the 1-skeleton $\Delta^{(1)}$ represents the region where the square intersects the three red lines, and the 0-skeleton $\Delta^{(0)}$ represents the origin where the three red lines intersect.

\begin{figure}[h]
    \centering
    \includegraphics[width=0.75\linewidth]{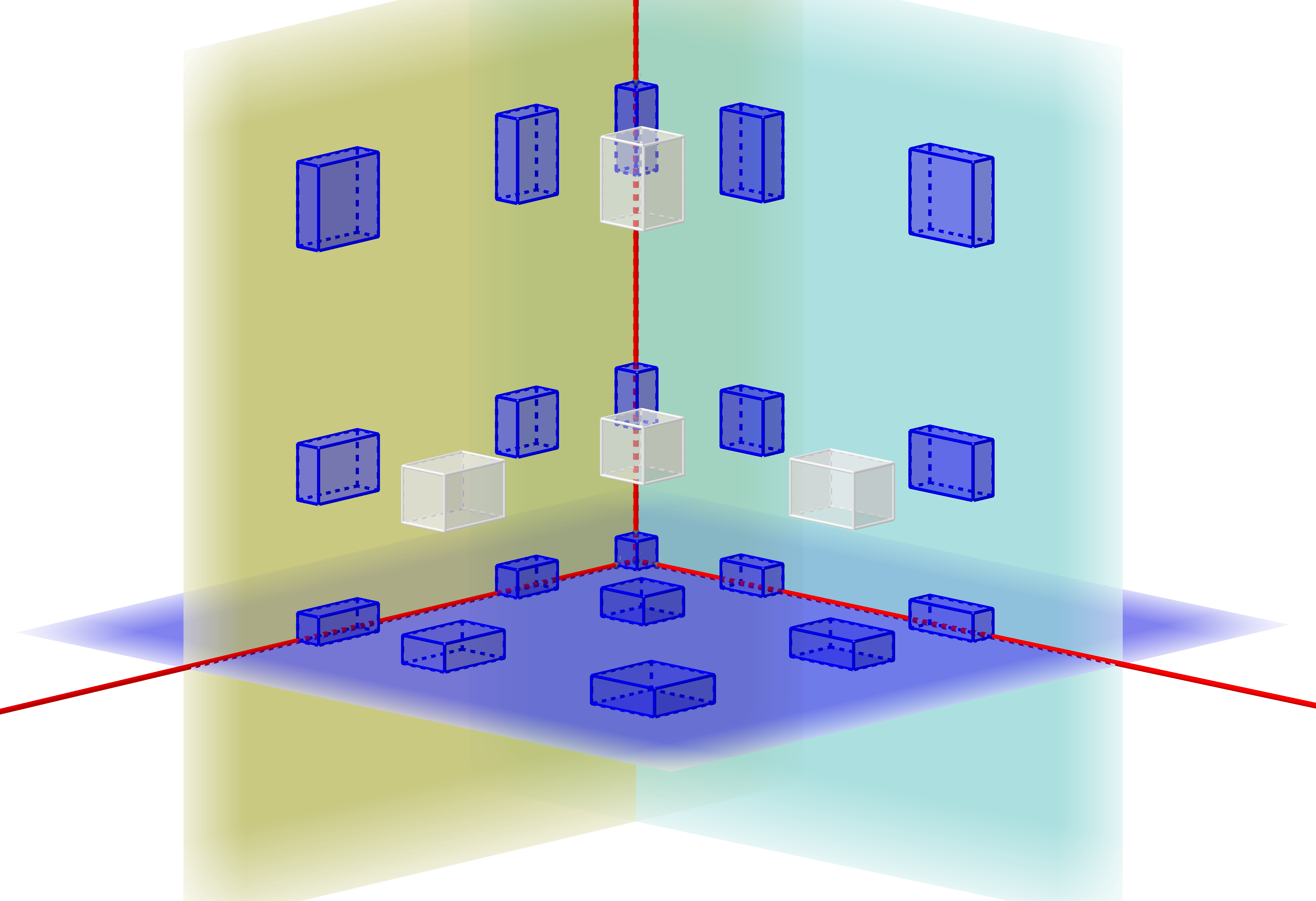}
    \caption{Products of three FCE spaces}
    \label{Products of three FCE spaces}
\end{figure}

The $(N-1)$-skeleton of $X$ is precisely the $1$-boundary of $X$, i.e., $\Delta^{(N-1)}=F_1$. Therefore, we now only need to prove the maximal coarse Baum-Connes conjecture for $\Delta^{(N-1)}$. We will accomplish this through induction. We claim that $\Delta^{(k+1)}$ admits a relative fibred coarse embedding into Hilbert space with respect to $\Delta^{(k)}$ for all $k\in\{0,\cdots,N-2\}$. Indeed, one can verify that $\Delta^{(k+1)}$ is a finite union of
\begin{equation}\label{eq: decomposition of skeleton}\prod_{j=1}^{N-(k+1)}X_{i_j,1}\times\prod_{i\in\{1,\cdots,N\}\backslash\{i_1,\cdots,i_{k+1}\}}X_i\end{equation}
where $\{i_1,\cdots,i_{k+1}\}$ should be taken as distinct from one another. The first part $\prod_{j=1}^{N-(k+1)}X_{i_j,1}$ is bounded, thus, the space defined as in \eqref{eq: decomposition of skeleton} is coarsely equivalent to its second part, which is a finite product of spaces with FCE structures. By Proposition \ref{pro: finite product of FCE spaces}, the space in \eqref{eq: decomposition of skeleton} admits a relative fibred coarse embedding into Hilbert space with respect to its boundary, which is a subspace of $\Delta^{(k)}$. Merging all these spaces as described in \eqref{eq: decomposition of skeleton} together, one can check that $\Delta^{(k+1)}$ admits a relative fibred coarse embedding into Hilbert space with respect to $\Delta^{(k)}$. By Theorem \ref{thm: reduction to infinity for maximal}, we have that the maximal coarse Baum-Connes conjecture for $\Delta^{(k+1)}$ is equivalent to the maximal coarse Baum-Connes conjecture for $\Delta^{(k)}$. By induction, it suffices to prove the maximal coarse Baum-Connes conjecture for $\Delta^{(1)}$. Note that $\Delta^{(1)}$ admits a relative fibred coarse embedding into Hilbert space with respect to $\Delta^{(0)}$, where
$$\Delta^{(0)}=\prod_{i=1}^NX_{i,1}$$
is bounded. Therefore, $\Delta^{(1)}$ admits a fibred coarse embedding into Hilbert space, whose maximal coarse Baum-Connes conjecture holds by \cite[Theorem 1.1]{CWY2013}. This finishes the proof.
\end{proof}

For the coarse Novikov conjecture, we also have the following corollary.

\begin{Cor}\label{cor: CNC for FCEtimesFCE}
For each $i\in\{1,\cdots, N\}$, assume $X_i$ is a sparse space with bounded geometry, i.e., $X_i=\bigsqcup_{n\in\IN}X_{i,n}$. If $X_i$ admits a fibred coarse embedding into Hilbert space (or a Hadamard manifold, or an $\ell^p$-space with $1\leq p<\infty$). Then the coarse Novikov conjecture holds for $\prod_{i=1}^NX_i$.
\end{Cor}

\begin{proof}[Proof of Corollary \ref{cor: CNC for FCEtimesFCE}]
By Theorem \ref{thm: CBC for RFCE}, Proposition \ref{pro: finite product of FCE spaces}, and Example \ref{exa: relative sparse}, it suffices to prove the coarse Novikov conjecture holds for $F_1$. Similar to Corollary \ref{cor: FCEH times FCEH}, it finally suffices to prove the coarse Novikov conjecture for $\Delta^{(1)}$, which admits a fibred coarse embedding into Hilbert space (Hadamard space, or an $\ell^p$-space with $1\leq p<\infty$). The Novikov conjecture for $\Delta^{(1)}$ holds by \cite{DGWY2025, GLWZ2024}, this finishes the proof.
\end{proof}

\subsection{On the sensitivity of the relative coarse Baum-Connes conjecture to subset selection}

In the last part of this paper, we shall exhibit an example that shows the sensitivity of the relative coarse Baum-Connes conjecture to subspace selection.

Let $Y=\bigsqcup_{m\in\IN}Y_m$ be a box space of $\mathbb F_2$ that is also a sequence of expander graphs. By \cite{CWY2013}, we know that $Y$ admits a fibred coarse embedding into Hilbert space. Since $Y$ is an expander, the coarse Baum-Connes assembly map is not surjective. Let $Z=\{n^2\mid n\in\IN\}$ be the coarse disjoint union of a sequence of single points. Define $X=Y\times Z$ to be the product space equipped with the $\ell^2$-product metric. We shall view $Y$ as a subset of $X$ via the canonical inclusion $Y= Y\times\{0\}\hookrightarrow X$. Fix a base point $y_0\in Y_0\subseteq Y$, we also view $Z$ as a subset of $X$ via $Z=\{y_0\}\times Z\hookrightarrow X$.


\begin{Pro}
The relative coarse Baum-Connes conjecture holds for $(X, Z)$. However, the relative coarse Baum-Connes assembly map is not surjective for $(X, Y)$.
\end{Pro}

\begin{proof}
First of all, it is not hard to show that $X$ is sparse with respect to both $Y$ and $Z$. By assumption, since $Y$ admits a fibred coarse embedding into Hilbert space, we shall show that $X$ admits a relative fibred coarse embedding into Hilbert space w.r.t. $Z$.

Indeed, take $(\H_y)_{y\in Y}$ to be the field of Hilbert space as in Definition \ref{def: FCE}. For any $R>0$, there exists $M\in\IN$ such that $(\H_w)_{w\in B(y,R)}$ admits a trivialization for all $y\in (\bigsqcup_{m=0}^NY_m)^c$. We define a field of real Hilbert spaces on $X$ by setting $\wt\H_{(y,n^2)}=\H_y\oplus\IR$ for any $(y,n^2)\in X$. The section $\wt s:X\to\bigsqcup \wt\H_x$ is defined to be
$$\wt s(y,n^2)=(s(y),n^2)\in\H_y\oplus\IR=\H_{(y,n^2)}.$$
For $R>0$, we take the same $N\in\IN$ as above, then for any $(y,n^2)\in (\bigsqcup_{m=0}^NY_m\times Z)^c$, one can define a trivialization of $(\wt\H_{(w,k^2)})_{(w,k^2)\in B((y,n^2),R)}$
$$t_{(y,n^2)}((w,k^2)):\wt\H_{(w,k^2)}\to\wt\H\quad\text{by }(v,t)\mapsto (t_y(w)(v),t).$$
From this construction, it is direct to check that this trivialization satisfies the two conditions in Definition \ref{def: partial FCE}. This shows that $X$ admits a relative fibred coarse embedding into Hilbert space w.r.t. $Z$. By Theorem \ref{thm: RCBC for RFCE spaces}, we prove the first part of this proposition.

For the second part, we denote $Y_{mn}=Y_m\times\{n^2\}$. Take $p_{mn}$ to be the average projection on $Y_{mn}$. Fix a rank one projection $q$ on $\ell^2(\IN)$. For $n=0$, it is known that $\oplus_{m\in\IN}p_{m0}\ox q$ determines an element in $K_0(C^*(Y))$ and this element can not be detected in $\lim_{d\to\infty}K_*(P_d(Y))$. One is referred to \cite[Section 13.2 \& Section 13.3]{WY2020} for detailed discussion.

Define $P=\oplus_{m,n\in\IN}p_{mn}\ox q$. With a similar argument with \cite[Lemma 13.2.10]{WY2020}, $P$ is a projection in $C^*(X)$. The matrix entries of $p_{mn}$ are all equal to $\frac{1}{\# Y_{m}}$. For fixed $m$, the matrix entries in $p_{mn}$ do not tend to $0$ as $n$ tends to $\infty$. Thus, $P$ is not in the ghostly ideal of $C^*(X)$ generated by $Y$. However, by definition, $P$ is in the ghostly ideal of $C^*(X)$ generated by $Z$. As a result, $[P]$ determines a non-zero element in $C^*_{Y,\infty}(X)$.

To see that the $K$-theory class of $[P]$ is also non-zero, we shall need some preparations. Set $W_m=Y_{m}\times Z$. The Roe algebra $C^*(W_m)$ can be seen as a subalgebra of $C^*(X)$. Moreover, $C^*(W_m)\cap I_G(Y)=\K(\ell^2(W_m),\H)$ since $Z$ has Property A. We define the algebra
$$\D=\frac{\prod_{m\in\IN}C^*(W_m)/\K}{\oplus_{m\in\IN}C^*(W_m)/\K}.$$
Set $C^*(X,Z)$ the geometric ideal generated by $Z$, then $C^*(X,Z)$ descends to an ideal of $C^*_{Y,\infty}(X)$ under the quotient map
$$\pi_{Y,\infty}: C^*(X)\to C^*_{Y,\infty}(X).$$
Then there exists a canonical $*$-homomorphism
$$\iota: \frac{C^*_{Y,\infty}(X)}{\pi_{Y,\infty}(C^*(X,Z))}\to \D$$
which is descent from the map $C^*(X)\to\prod_{m\in\IN}C^*(W_m)$ defined by $T\mapsto\oplus_{m\in\IN}\chi_{W_m}T\chi_{W_m}$. It is straightforward to check that this map is well-defined and induces a map on the level of $K$-theory
$$tr:=\iota_*\circ(\pi_{Y,\infty})_*: K_*(C^*_{Y,\infty}(X))\to K_*(\D)\cong \frac{\prod_{m\in\IN}K_*(C^*(W_m)/\K)}{\oplus_{m\in\IN}K_*(C^*(W_m)/\K)},$$
where the composition of these two maps is denoted by $tr$ as the ``trace'' map.
Since each $W_m$ is coarsely equivalent to $Z$, we can compute that $K_*(C^*(Z)/\K)\cong \frac{\prod_{n\in\IN}\IZ}{\oplus_{n\in\IN}\IZ}$. Denote by $[1]\in K_*(C^*(Z)/\K)$ the element $(1,1,1,\cdots)$. It is direct to check that $tr([P])=[[1],[1],\cdots]$ which is non-zero, thus $[P]\in K_0(C^*_{Y,\infty}(X))$ is never zero. Moreover, the element $[P]$ can never come from $K_*(\pi_{Y,\infty}(C^*(X,Z)))$ since the map $tr$ will take any element in this group to $0$. One can compare the process above with \cite[Lemma 13.3.8]{WY2020} and \cite[Theorem 5.5]{FSN2014}.

We have the following commuting diagram:
\[\begin{tikzcd}
\lim\limits_{d\to\infty}K_*(\pi_{Y,\infty}(C^*_{L,Z}(P_d(X)))) \arrow[d,"\mu"] \arrow[r,"i"] & \lim\limits_{d\to\infty}K_*(C^*_{L,Y,\infty}(P_d(X))) \arrow[r]\arrow[d,"\mu_{Y,\infty}"] & \lim\limits_{d\to\infty}K_*(C^*_{L,Y\cup Z,\infty}(P_d(X))) \arrow[d,"\mu_{Y\cup Z,\infty}","\cong"'] \\
K_*(\pi_{Y,\infty}(I_G(Z))) \arrow[r]           & K_*(C^*_{Y,\infty}(X))  \arrow[r,"\pi_Z"]     &   K_*(C^*_{Y\cup Z,\infty}(X))
\end{tikzcd}\]
Both the upper and the lower rows are exact; one should compare this with \cite{HLS2002}. The upper row is essentially an exact sequence on the classifying space for groupoids. The map $\mu_{Y\cup Z,\infty}$ is an isomorphism by Theorem \ref{thm: RCBC for RFCE spaces}. If $\mu_{Y,\infty}$ is surjective, then there should exist $a\in\lim_{d\to\infty}K_*(C^*_{L,Y,\infty}(P_d(X)))$ such that $\pi_{Y,\infty}(a)=[P]$. Notice that $\pi_Z([P])=0$, by diagram chasing, there should be some element $b\in \lim_{d\to\infty}K_*(\pi_{Y,\infty}(C^*_{L,Z}(P_d(X))))$ such that $i(b)=a$. However, as we discussed above, an element that comes from $C^*_{L,Z}(P_d(X))$ has to take an index in $K_*(C^*(X,Z))$, whose image under the map $tr$ should be $0$. This leads to a contradiction that $tr([P])\ne 0$.
\end{proof}

\bibliographystyle{alpha}
\bibliography{ref}

\end{document}